\documentclass{IEEEojcsys}

\usepackage[colorlinks,urlcolor=blue,linkcolor=blue,citecolor=blue]{hyperref}

\usepackage{array}
\usepackage[dvipsnames]{xcolor}

\usepackage{graphicx}
\usepackage{psfrag}
\usepackage{tikz,pgfplots}
\usetikzlibrary{shapes.geometric, positioning, fit}
\usepackage{cite}
\usepackage{listings}
\usepackage{amsmath,amssymb,amsfonts}
\usepackage{mathtools}
\usepackage{graphicx}
\usepackage{amsthm}
\usepackage{enumerate}
\usepackage{algpseudocode}
\usepackage{algorithm}
\usepackage{bbold}
\usepackage{subcaption}
\usepackage{svg}
\usepackage{caption}
\usepackage{multicol}
\usepackage{booktabs}
 \usepackage{multirow}
\usepackage{textcomp}
\usepackage{stfloats}
\usepackage{epsfig}
\usepackage{pgfplots}
  \pgfplotsset{compat=newest}
  \usetikzlibrary{plotmarks}
  \usetikzlibrary{calc}
\usepackage{tikz}
\usepackage{circuitikz}  
\usetikzlibrary{positioning}
  \usetikzlibrary{arrows,shapes, positioning}
  \usepgfplotslibrary{patchplots}
  \usepackage{grffile}
  \usepackage{amsmath}
\usepackage{hyperref}

\newcommand{\Rset}{\mathbb{R}}

\newcommand{\Uset}{\mathbb{U}}

\newcommand{\Yset}{\mathbb{Y}}

\newcommand{\bu}{\mathbf{u}}

\newcommand{\by}{\mathbf{y}}
\newcommand{\bg}{\mathbf{g}}

\usepackage{tikz}

\definecolor{Matlab_blue}{rgb}{0, 0.4470, 0.7410}
\definecolor{Matlab_orange}{rgb}{0.8500, 0.3250, 0.0980}
\definecolor{Matlab_yellow}{rgb}{0.9290 0.6940 0.1250}
\definecolor{green_dark}{rgb}{0.0, 0.5, 0.0}
\definecolor{DeePONet_purple}{rgb}{0.701961,0.701961,1.00000}%
\definecolor{DeePONet_gray}{rgb}{0.850980,0.8509801,0.850980}%

\definecolor{codegreen}{rgb}{0,0.6,0}
\definecolor{codegray}{rgb}{0.5,0.5,0.5}
\definecolor{codepurple}{rgb}{0.58,0,0.82}
\definecolor{backcolour}{rgb}{0.95,0.95,0.92}
\definecolor{darkgreen}{rgb}{0,0.5,0}

\lstdefinestyle{mystyle}{
    backgroundcolor=\color{backcolour},   
    commentstyle=\color{codegreen},
    keywordstyle=\color{magenta},
    numberstyle=\tiny\color{codegray},
    stringstyle=\color{codepurple},
    basicstyle=\ttfamily\footnotesize,
    breakatwhitespace=false,         
    breaklines=true,                 
    captionpos=b,                    
    keepspaces=true,                 
    numbers=left,                    
    numbersep=5pt,                  
    showspaces=false,                
    showstringspaces=false,
    showtabs=false,                  
    tabsize=2
}

\newcommand{\col}{\operatorname{col}}
\newcommand{\diag}{\operatorname{diag}}

\newtheorem{theorem}{Theorem}
\newtheorem{problem}{Problem}
\newtheorem{lemma}{Lemma}
\newtheorem{proposition}{Proposition}

\newtheorem{remark}{Remark}

\begin{document}

\raggedbottom 

\sptitle{Article Category}

\title{Deep Operator Neural Network \\ Model Predictive Control}

\author{T.O. de Jong\affilmark{1} (Student Member, IEEE)}

\author{K. Shukla\affilmark{2}}

\author{M. Lazar\affilmark{1} (Senior Member, IEEE)}

\affil{Eindhoven University of Technology, The Netherlands}
\affil{Brown University, The United States}

\corresp{CORRESPONDING AUTHOR: T.O.~de~Jong (e-mail: \href{mailto:t.o.d.jong@tue.nl}{t.o.d.jong@tue.nl})}

\markboth{Deep Operator Neural Network Model Predictive Control}{T.O. de Jong {\itshape et al}.}

\begin{abstract}
In this paper, we consider the design of model predictive control (MPC) algorithms based on deep operator neural networks (DeepONets) \cite{DeepONet}. 
These neural networks are capable of accurately approximating real-and complex-valued solutions \cite{jiang2024complex} of continuous-time nonlinear systems without relying on recurrent architectures. The DeepONet architecture is made up of two feedforward neural networks: the branch network, which encodes the input function space, and the trunk network, which represents dependencies on temporal variables or initial conditions. Utilizing the original DeepONet architecture \cite{DeepONet} as a predictor within MPC for Multi-Input-Multi-Output (MIMO) systems requires multiple branch networks, to generate multi-output predictions, one for each input. Moreover, to predict multiple time steps into the future, the network has to be evaluated multiple times. Motivated by this, we introduce a multi-step DeepONet (MS-DeepONet) architecture that computes in one-shot multi-step predictions of system outputs from multi-step input sequences, which is better suited for MPC. We prove that the MS-DeepONet is a universal approximator in terms of multi-step sequence prediction. Additionally, we develop automated hyper parameter selection strategies and implement MPC frameworks using both the standard DeepONet and the proposed MS-DeepONet architectures in PyTorch. The implementation is publicly available on \href{https://github.com/todejong/Deep-Operator-Neural-Network-Model-Predictive-Control.git}{GitHub}. Simulation results demonstrate that MS-DeepONet consistently outperforms the standard DeepONet in learning and predictive control tasks across several nonlinear benchmark systems: the van der Pol oscillator, the quadruple tank process, and a cart–pendulum unstable system, where it successfully learns and executes multiple swing-up and stabilization policies.
\end{abstract}

\begin{IEEEkeywords}
Model predictive control, Deep operator neural networks, Nonlinear dynamical systems, Multi-step prediction, Constrained control.
\end{IEEEkeywords}

\maketitle

\section{INTRODUCTION}
\label{sec1}
Merging the strong expressivity of neural networks (NNs) with model predictive control (MPC) is an effective approach for control design in the case of complex nonlinear systems, which are difficult to model using first principles alone. Early works in neural predictive control used fully-connected feedforward neural networks to construct nonlinear autoregressive exogenous (NARX) type of one-step models, which were then simulated to obtain multi-step output predictions, see, for example, \cite{lazar2002neural, lawrynczuk2009neural, Norgaard_2000}, and the references therein. To reduce the growth of prediction errors caused by one-step modeling errors, in \cite{lawrynczuk2009neural, masti2020learning} separate input-output predictors for each time step have been introduced. These predictors can be trained jointly or separately. The structured multi-step predictor architecture is suitable for minimizing the multi-step prediction error during the training performed offline and for online predictive control, as future system outputs can be directly inferred without simulating one-step models. Recently, inspired by the deep prediction network  architecture \cite{dalla2022deep} for system identification, \cite{Deep_Prediction_Nets_Lazar} classified existing prediction network architectures and developed a new deep subspace architecture with improved performance and robustness to noisy data.

Stimulated by recent developments in recurrent neural networks architectures suited for sequence modeling, such as Long-short Term Memory (LSTM) \cite{hochreiter1997long} and Gated recurrent Unit (GRU) \cite{cho2014learning} based neural networks, recent works in neural predictive control focused on adopting such advanced architectures and providing formal guarantees for stability shown in research articles \cite{Bonassi2021, Zarzycki2022, Lalo_2024, Irene_LTSM, Masero2024, Schimperna2024 } and the survey papers \cite{Bonassi_2022, Survey_NN_MPC}. Such recurrent neural networks demonstrate superior modeling performance, as they can reproduce nonlinear state-space models, but are more challenging to train. Also, using recurrent neural networks within predictive control requires careful initialization of the network internal state, which is non-trivial, i.e., arbitrary initialization can lead to transients and loss of prediction accuracy. The popular framework of physics-informed neural networks (PINNs) has also been recently adopted in predictive control, see, for example, \cite{PINNs_MPC, Ramp_net_MPC_PINNs}. In these works, PINNs are used to learn one-step prediction models that are simulated to obtain future system output over a finite prediction horizon. Other works have also addressed the real-time implementation aspects of neural predictive control, see, for example, \cite{Realtime_Neural_MPC}.

In conventional regression tasks, neural networks learn a function mapping between two finite-dimensional spaces. Recently, two approaches—Deep Operator Network (DeepONet) \cite{DeepONet} and Fourier Neural Operator \cite{li2020fourier}-based architectures—have been proposed to model the mapping between infinite-dimensional function spaces. The  DeepONet architecture \cite{DeepONet} is inspired by \cite{Chen_ONN}, and it provides a universal approximator of linear and nonlinear operators. The architecture of DeepONet consists of a branch deep neural network (typically a feedforward network) and a trunk neural network (typically a feedforward network) which are combined to learn a real-valued operator, e.g., the solution of a nonlinear dynamical system. For instance, in the case of a time-dependent differential operator, the branch network receives inputs as a function evaluated at various time instances, while the trunk network takes as input the time instances along with the initial conditions or states. The final output of a DeepONet is produced by combining the outputs of the branch and trunk networks along the dimension of their last layers. The training of the DeepONet is carried out through a carefully constructed loss function that measures the difference between the DeepONet predictions and the labeled data. Once the training is complete, the solution to a linear or nonlinear system is inferred by performing a forward pass of the DeepONet using unseen initial conditions, desired time instants and control inputs. 

DeepONets offer an attractive alternative to NARX-type feedforward neural networks and state-space recurrent neural networks, as the training process is similar to that of feedforward neural networks, but the learning accuracy is virtually unlimited for solutions of continuous nonlinear systems operating on compact domains. So far, DeepONet has been successfully applied to various engineering and scientific problems (e.g., \cite{shukla2024deep, oommen2022learning, shukla2024deepII, jiang2024complex, cai2021deepm}) and to control of partial differential equations and time-delay systems \cite{Miroslav_2024, Miroslav_2024_Auto, Miroslav_2025_predictor}. In \cite{Miroslav_2025_predictor}, in particular, the authors use a DeepONet to predict the state of the system at a future moment in time, in order to stabilize a differential equation with input delay. However, to the best of our knowledge, the implementation of the DeepONet  in model predictive control has not yet been investigated.

In this paper, we therefore explore the potential of employing the DeepONet in predictive control. The first challenge deals with the fact that multiple-input-multiple-output (MIMO) nonlinear systems necessitate an extension of the DeepONet architecture to handle MIMO operators. Additionally, we aim to develop multi-step DeepONet predictors that can forecast future system outputs without relying on simulation. Existing extensions of the DeepONet to multiple inputs, such as MiONet \cite{MiONet} and MimoONet \cite{MimoONet}, address different problems of learning operators with multiple input functions and outputs, and involve multiple branch networks, which complicate the architecture and increase its complexity. This is not ideal for predictive control, where an effective predictor architecture is crucial.

The first contribution of this work is the development of the multistep DeepONet (MS-DeepONet) architecture, an adaptation of the traditional DeepONet designed to predict the behavior of the system at multiple future time points in one shot (i.e., one network evaluation). Given an initial state $x(t_0)$ and a discrete input sequence $\{u(t_0), \dots, u(t_{N-1})\}$, the model predicts the corresponding output sequence $\{y(t_1), \dots, y(t_N)\}$. This structure enables efficient computation and makes the model well-suited for online use in model predictive control (MPC), where fast and accurate multi-step predictions are essential. By suitably adapting the results of \cite{Chen_ONN}, we prove that the developed MS-DeepONet is a universal multi-step approximator. Hence, we prove that multi-step output sequences generated by multi-step input sequences can be approximated with arbitrary accuracy. 

The second major contribution is the derivation of an equivalent basis representation of MS-DeepONet, enabling a more efficient MPC formulation. This perspective builds on prior work interpreting the final hidden layer of a neural network as an implicit adaptive basis \cite{Widrow_2013}. The resulting basis interpretation highlights both computational and structural advantages of MS-DeepONet over conventional architectures in predictive control. It also facilitates the design of data-enabled predictive controllers based on DeepONet, extending concepts previously applied to feedforward networks \cite{Neural_Basis_Lazar}.

The final contribution of this work is the development of optimal hyper parameter selection algorithms for both the standard DeepONet and the proposed MS-DeepONet for learning dynamical systems form data. In addition, we design MPC schemes that incorporate both the conventional MIMO-DeepONet and MS-DeepONet architectures. All learning algorithms are implemented in PyTorch \cite{PyTorch2019}, with corresponding learning algorithms and MPC implementations provided in Python. 

After introducing several notations and basic definitions, the remainder of this paper is structured as follows. Section~\ref{sec2} provides the preliminary results of DeepONet and the problem statement. In Section~\ref{sec3}, we introduce the multi-step DeepONet architecture and present the equivalent basis representations for both the standard unstacked DeepONet and the multi-step DeepONet. We also present a universal approximation theorem for the MS-DeepoNet for learing solutions of continuous--time dynamical systems, given that the input signal is a piecewise constant finite length signal. Section~\ref{sec:ablation_study} presents learning algorithms for the MS-DeepONet, alongside an ablation study designed to identify the optimal hyper parameters of the network. In Section~\ref{sec6}, we demonstrate how the DeepONet and MS-DeepONet architectures can be integrated as prediction models within MPC algorithms. Section~\ref{sec:examples} presents illustrative examples to validate the proposed approaches, and Section~\ref{sec:conclusions} concludes with a summary of results and future work.

\paragraph*{Notation and basic definitions}
For any finite collection of \( q \in \mathbb{N}_{\geq 1} \) vectors \( \xi_1, \dots, \xi_q \in \mathbb{R}^{n_1} \times \dots \times \mathbb{R}^{n_q} \), we will make use of the operator $\text{col}(\xi_1, \dots, \xi_q) := \begin{bmatrix} \xi_1^\top & \dots & \xi_q^\top \end{bmatrix}^\top$ to indicate the vector stacking them in a column.  For two matrices $A \in \Rset^{n\times m}$, $B \in \Rset^{p \times q}$, we denote their Kronecker product by $A\otimes B$, defined as
$$
A \otimes B = \begin{bmatrix}
    a_{1,1} B & \dots & a_{1,m} B\\
    \vdots & & \vdots \\
    a_{n,1} B & \dots & a_{n,m} B
\end{bmatrix}.
$$
Moreover, we denote the vectorization of a matrix $A =\begin{bmatrix}
    a & b\\
    c & d
\end{bmatrix}$ by $\text{vec}({A}) = \col(a,c,b,d)$. A Hankel matrix constructed using an arbitrary dataset $\mathbf{z}$, with $N$ rows and $T$ columns starting at index $i$ is denoted as:
\begin{align*}
&\mathcal{H}^i_{[N,T]}(\mathbf{z}) := \\
&\begin{bmatrix}
    z(i) & z(i+1) & \dots & z(i+T-1) \\
    z(i+1) & z(i+2) & \dots & z(i+T) \\
    \vdots & \vdots & & \vdots \\
    z(i+N-1) & z(i+N) & \dots & z(i+T+N-2) \\
\end{bmatrix},
\end{align*}
where the samples in $\mathbf{z}$ can comprise information from multiple channels (e.g., multiple inputs or outputs), i.e., $z(i) = \col(z_1(i), z_2(i), \dots,z_r(i))$. For any vector $\mathbf{z}\in\mathbb{R}^q$ we denote by $[\mathbf{z}]_i$ the $i^{\text{th}}$ element of the vector and $[\mathbf{z}]^j_i$ is a vector containing the $i^{\text{th}}$ up to and including the $j^{\text{th}}$ element for all $i =1,\dots,q$. By $C(K)$ we denote the Banach space with norm $\|\cdot\|_{X}$ and $C(K)$ is the Banach space of all continuous functions defined on $K$ with $\|f\|_{C(K)}=\max_{x\in K}|f(x)|$. If a function $g:\mathbb{R}\rightarrow \mathbb{R}$ (continuous or discontinuous) satisfies that all the linear combinations $\sum_{i=1}^Nc_ig(\lambda_ix+\theta_i), \lambda_i\in\mathbb{R}$, $\theta_i\in\mathbb{R}$, $c_i\in\mathbb{R}$, $i=1,\dots,N$ are dense in every $C([a, b])$, then $g$ is called a Tauber-Wiener (TW) function.
\section{PRELIMINARIES}
\label{sec2}
We consider continuous-time nonlinear MIMO systems with inputs $u\in\mathbb{U}\subset \mathbb{R}^{n_u}$, measured states and outputs $x\in\mathbb{X}\subset \mathbb{R}^{n_x}$, $y\in\mathbb{Y}\subset \mathbb{R}^{n_y}$ where $\mathbb{U}$, $\mathbb{X}$ and $\mathbb{Y}$ are compact sets and time $t\in\mathbb{R}_+$, i.e.,
\begin{equation}\label{eqn:systems}
    \begin{aligned}
         \Dot{x}(t) &= f(x(t),u(t)), \\
        y(t) &= h(x(t)),
    \end{aligned}
\end{equation}
In the case of single-input single-output (SISO) systems, such a dynamical system can be regarded as an operator that maps input functions to output functions. In this regard,  let $G$ be an operator taking an input function $u$, and then $G(u)$ is the corresponding output function. For any point $z$ in the domain of $G(u)$, the output $G(u)(z) \in \Rset$ is a real number, where in the case of continuous-time systems of the form \eqref{eqn:systems} the variable $z:= \col(x(0),t)$ would consist of the initial condition $x(0)$ and time $t$. Our goal is to learn models for such operators from data. 
\begin{figure}[t!]
        \centering
        \begin{tikzpicture}[
Block_purp/.style={draw=black, fill = DeePONet_purple, very thick, minimum width=0.5cm, minimum height=0.5cm},
Blockwhite/.style={draw=black, fill = white , very thick, minimum width=0.5cm, minimum height=0.5cm},
Circ_gray_prod/.style={draw=black, fill=DeePONet_gray, shape=circle, thick, minimum width=0.7cm, minimum height=0.7cm, inner sep=0pt, outer sep=0pt},
node distance=10mm  
]

\tikzset{Circ_gray/.style={draw=black, fill=DeePONet_gray, shape=circle, thick, minimum width=0.7cm, minimum height=0.7cm, inner sep=0pt, outer sep=0pt}}

\tikzset{Circ_purp/.style={draw=black, fill=DeePONet_purple, shape=circle, thick, minimum width=0.7cm, minimum height=0.7cm, inner sep=0pt, outer sep=0pt}}

\node[draw, fill=white, very thick, minimum width=1cm, minimum height=3cm] (block_1) at (0, 0) {};

\node[draw, fill=white, very thick, minimum width=1cm, minimum height=3.85cm] (block_2) at (-1.5cm, 0) {};

\node[Circ_gray, anchor=center] (b1) at (block_1.north) [yshift=-0.45cm] {{\small $y_1$}};  
\node[Circ_gray, anchor=center] (b2) at (block_1.north) [yshift=-1.25cm] {{\small $y_2$}};  
\node[anchor=center] (ellipsis) at (block_1.north) [yshift=-1.8cm] {{\footnotesize \vdots}};  
\node[Circ_gray, anchor=center] (bp) at (block_1.north) [yshift=-2.55cm] {{\small $y_p$}};  

\node[anchor=center] (text) at (block_2.north) [yshift=+0.45cm] {{\footnotesize \textcolor{blue}{$\phi(\mathbf{u}_{\text{NN}})$}}};
\node[Circ_purp, anchor=center] (h44) at (block_2.north) [yshift=-0.45cm] {{\small $\phi^1$}};  
\node[Circ_purp, anchor=center] (h43) at (block_2.north) [yshift=-1.25cm] {{\small $\phi^2$}};  
\node[anchor=center] (ellipsis_2) at (block_2.north) [yshift=-1.8cm] {{\footnotesize \vdots}};  
\node[Circ_purp, anchor=center] (h42) at (block_2.north) [yshift=-2.55cm] {{\small $\phi^{n-1}$}};  
\node[Circ_purp, anchor=center] (h41) at (block_2.north) [yshift=-3.35cm] {{\small $\phi^n$}};

\node[anchor=center] (h34) at (block_2.north) [xshift=-1.0cm, yshift=-0.45cm] {{\footnotesize \dots}};  
\node[anchor=center] (h33) at (block_2.north) [xshift=-1.0cm, yshift=-1.25cm] {{\footnotesize \dots}};  
\node[anchor=center] (ellipsis_2) at (block_2.north) [xshift=-1.0cm, yshift=-1.8cm] {{\footnotesize \dots}};  
\node[anchor=center] (h32) at (block_2.north) [xshift=-1.0cm, yshift=-2.55cm] {{\footnotesize \dots}};  
\node[anchor=center] (h31) at (block_2.north) [xshift=-1.0cm, yshift=-3.35cm] {{\footnotesize \dots}};

\node[Circ_purp, anchor=center] (h24) at (block_2.north) [xshift=-2.0cm, yshift=-0.45cm] { };  
\node[Circ_purp, anchor=center] (h23) at (block_2.north) [xshift=-2.0cm, yshift=-1.25cm] { };  
\node[anchor=center] (ellipsis_2) at (block_2.north) [xshift=-2.0cm, yshift=-1.8cm] {{\footnotesize \vdots}};  
\node[Circ_purp, anchor=center] (h22) at (block_2.north) [xshift=-2.0cm, yshift=-2.55cm] { };  
\node[Circ_purp, anchor=center] (h21) at (block_2.north) [xshift=-2.0cm, yshift=-3.35cm] { };

\node[Circ_purp, anchor=center] (h13) at (block_1.north) [xshift=-5.0cm, yshift=-0.45cm] { };  
\node[Circ_purp, anchor=center] (h12) at (block_1.north) [xshift=-5.0cm, yshift=-1.25cm] { };  
\node[anchor=center] (ellipsis_2) at (block_1.north) [xshift=-5.0cm, yshift=-1.8cm] {{\footnotesize \vdots}};  
\node[Circ_purp, anchor=center] (h11) at (block_1.north) [xshift=-5.0cm, yshift=-2.55cm] { };  

\node[anchor=center] (in) at (block_1.north) [xshift=-6.0cm, yshift=-1.4cm] {$\mathbf{u}_{\text{NN}}$};  

\node[anchor=center] (out) at (block_1.north) [xshift=0cm, yshift=0.5cm] {$\mathbf{y}_{\text{NN}}$};  

\draw[-, thick] (h11.east) -- (h21.west); 
\draw[-, thick] (h11.east) -- (h22.west);
\draw[-, thick] (h11.east) -- (h23.west);
\draw[-, thick] (h11.east) -- (h24.west);

\draw[-, thick] (h12.east) -- (h21.west); 
\draw[-, thick] (h12.east) -- (h22.west);
\draw[-, thick] (h12.east) -- (h23.west);
\draw[-, thick] (h12.east) -- (h24.west);

\draw[-, thick] (h13.east) -- (h21.west); 
\draw[-, thick] (h13.east) -- (h22.west);
\draw[-, thick] (h13.east) -- (h23.west);
\draw[-, thick] (h13.east) -- (h24.west);

\draw[-, thick] (h41.east) -- (b1.west); 
\draw[-, thick] (h41.east) -- (b2.west); 
\draw[-, thick] (h41.east) -- (bp.west); 

\draw[-, thick] (h42.east) -- (b1.west);
\draw[-, thick] (h42.east) -- (b2.west);
\draw[-, thick] (h42.east) -- (bp.west);

\draw[-, thick] (h43.east) -- (b1.west);
\draw[-, thick] (h43.east) -- (b2.west);
\draw[-, thick] (h43.east) -- (bp.west);

\draw[-, thick] (h44.east) -- (b1.west);
\draw[-, thick] (h44.east) -- (b2.west);
\draw[-, thick] (h44.east) -- (bp.west);

\end{tikzpicture}
        \caption{Illustration of a fully connected feedforward neural network (FFN) and it's underlying basis.}
        \label{fig:FFN}
\end{figure}
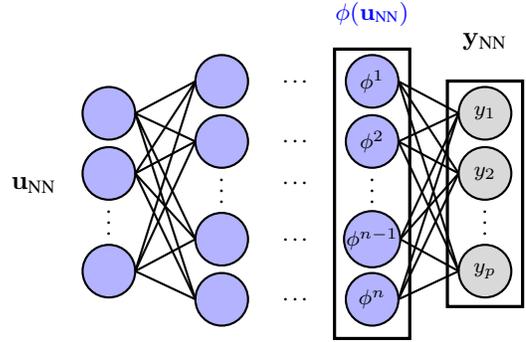
A powerful method for learning functions from data is to use fully connected feedforward neural networks, see \figurename~\ref{fig:FFN}. The universal approximation theorem claims that standard multilayer feedforward networks with a single hidden layer that contains a finite number of hidden neurons, and with an arbitrary activation function are universal approximators of continuous functions \cite{hornik1991approximation, csaji2001approximation}. Letting $\mathbf{u}_{NN}$ and $\mathbf{y}_{NN}$ represent the inputs and outputs of a generic deep NN, the output of the NN can be expressed as \cite{Neural_Basis_Lazar}:
$$
\mathbf{y}_{NN} = W_o \phi(\mathbf{u}_{NN}) + b_o,
$$
where $\phi=\col(\phi^1,\ldots,\phi^n)$, collects the outputs of the last hidden layer and $W_o\in\mathbb{R}^{p\times n}$ and $b_o\in\mathbb{R}^p$. Since the outputs of the network are in the affine span of $\phi$, this can be viewed as an implicit underlying basis. Based on this insight, a single deep NN can be used to learn functions with multiple inputs and outputs. Such networks were used to design a multistep and multi-output predictor for discrete--time dynamical systems in \cite{Neural_Basis_Lazar}. The prediction models were implemented in a predictive control scheme to actively control the dynamical system. However, such solutions are not directly applicable to learn models for a continuous--time dynamical system \eqref{eqn:systems} since this is an operator rather than a function.

DeepONet was proposed in \cite{DeepONet} for learning operators $G(u)(z)$ from data. In general, the network takes inputs composed of two parts: $u$ and $z$, and outputs $G(u)(z)\in\mathbb{R}$. The DeepONet consist of two FFNs, namely the branch and trunk network, see \figurename~\ref{fig:DeePONet} and \figurename~\ref{fig:DeePONet_unstacked} for the stacked and unstacked formulations respectively. This was later extended to multi input multi output (MIMO) operators in \cite{MimoONet} by adding branch networks for each input signal, see \figurename~\ref{fig:MimoONet}. In this work, we propose using the DeepONet architecture to design predictive controllers for continuous--time MIMO systems of the form \eqref{eqn:systems}. At each discrete--time step \( t = kT_s \), it is necessary to generate multiple predictions of the system's output over time, specifically let \( \{ y((k+1)T_s), y((k+2)T_s), \dots, y((k+N)T_s) \} \), where \( T_s \in \mathbb{R} \) denotes the sampling time, and \( N \in \mathbb{N} \) represents the prediction horizon. In a predictive control setting we are typically interested in finite length signals where only sampled measurements of the output $y(t)$ of system \eqref{eqn:systems} are available and the input signal $u(t)$ is implemented  by using Zero-Order-Hold (ZOH), i.e.,
\begin{equation}\label{eqn:piecewise_constant_input}
   u(t) = u_k \quad \text{for} \quad t \in \left[k T_s, (k+1) T_s\right), 
\end{equation}
where $u_k\in\mathbb{R}^{n_u}$ is the constant input of the signal during the $k$-th interval, and $k \in \mathbb{N}$. We denote by $\bu_k = \col(u_k, \dots, u_{k+N-1})$ and the output sequence $\by_k = \col(y((k+1)T_s), \dots, y((k+N)T_s))$. 
\begin{figure*}[t!]
    \centering
    \begin{subfigure}[t]{0.48\textwidth}
        \centering
    \begin{tikzpicture}[
Block_purp/.style={draw=black, fill = DeePONet_purple, very thick, minimum width=2cm, minimum height=0.5cm},
Blockwhite/.style={draw=black, fill = white , very thick, minimum width=0.5cm, minimum height=0.5cm},
Circ_gray_prod/.style={draw=black, fill=DeePONet_gray, shape=circle, thick, minimum width=0.7cm, minimum height=0.7cm, inner sep=0pt, outer sep=0pt},
node distance=10mm  
]
\tikzset{Circ_gray/.style={draw=black, fill=DeePONet_gray, shape=circle, thick, minimum width=0.7cm, minimum height=0.7cm, inner sep=0pt, outer sep=0pt}}

\node[draw, fill=white, very thick, minimum width=1cm, minimum height=3cm] (block_1) at (0, 0) {};

\node[Circ_gray, anchor=center] (b1) at (block_1.north) [yshift=-0.45cm] {{\footnotesize $b_1$}};  
\node[Circ_gray, anchor=center] (b2) at (block_1.north) [yshift=-1.25cm] {{\footnotesize $b_2$}};  
\node[anchor=center] (ellipsis) at (block_1.north) [yshift=-1.8cm] {{\footnotesize \vdots}};  
\node[Circ_gray, anchor=center] (bp) at (block_1.north) [yshift=-2.55cm] {{\footnotesize $b_p$}};  
    
\node[draw, fill=white, very thick, minimum width=1cm, minimum height=3cm] (block_2) at (0, -3.3cm) {};

\node[Circ_gray, anchor=center] (t1) at (block_2.north) [yshift=-0.45cm] {{\footnotesize $t_1$}};  
\node[Circ_gray, anchor=center] (t2) at (block_2.north) [yshift=-1.25cm] {{\footnotesize $t_2$}};  
\node[anchor=center] (ellipsis) at (block_2.north) [yshift=-1.8cm] {{\footnotesize \vdots}};  
\node[Circ_gray, anchor=center] (tp) at (block_2.north) [yshift=-2.55cm] {{\footnotesize $t_p$}};  

\node[Block_purp] (branch1) [left=0.5cm of b1] {\parbox{2.0cm}{\centering Branch net$_1$}};
\node[Block_purp] (branch2) [left=0.5cm of b2] {\parbox{2.0cm}{\centering Branch net$_2$}};
\node[Block_purp] (branch3) [left=0.5cm of bp] {\parbox{2.0cm}{\centering Branch net$_p$}};

\node[Block_purp] (trunk1) [left=0.5cm of block_2] {\parbox{2.0cm}{\centering Trunk net}};

\node[Circ_gray_prod] (circ1) [right=of $(block_1)!0.5!(block_2)$, yshift=0mm, xshift=0.5mm, inner sep=0pt, outer sep=0pt] {$\times$};


\node[Blockwhite] (block4) [left=3.0cm of block_1] {\parbox{0.8cm}{\small \centering ${\begin{matrix}
    u(s_1) \\
    u(s_2) \\
    \vdots \\
    u(s_m)
\end{matrix}}$}};

\node (input_trunk) at (trunk1) [shift={(-2.0cm,0.0cm)}] {\small ${z}$};

 \draw[->, thick] (block4.east) ++(0, 5mm) -- (branch1.west) node[midway, above] {};
  \draw[->, thick] (block4.east) -- (branch2.west) node[midway, above] {};
 \draw[->, thick] (block4.east) ++(0, -5mm) -- (branch3.west) node[midway, above] {};

 \draw[->, thick] (branch1.east) -- (b1.west) node[midway, above] {};
 \draw[->, thick] (branch2.east) -- (b2.west) node[midway, above] {};
 \draw[->, thick] (branch3.east) -- (bp.west) node[midway, above] {};

  \draw[->, thick] (trunk1.east) -- (t1.west) node[midway, above] {};
 \draw[->, thick] (trunk1.east) -- (t2.west) node[midway, above] {};
 \draw[->, thick] (trunk1.east) -- (tp.west) node[midway, above] {};
 
\draw[->, thick] (block_1.east) -- (circ1.north west) node[midway, above] {};
\draw[->, thick] (block_2.east) -- (circ1.south west) node[midway, above] {};

\draw[->, thick] (input_trunk.east) -- (trunk1.west) node[midway, above] {};

\draw[->, thick] (circ1.east) --  ++ (1cm,0cm) node[midway, above, yshift=0.3cm, xshift = 0.2cm] {{\small$G(u)(z)$}};

\end{tikzpicture}
        \caption{Traditional stacked DeepONet architecture.}
        \label{fig:DeePONet}
    \end{subfigure}
    \hfill
    \begin{subfigure}[t]{0.48\textwidth}
        \centering
    \begin{tikzpicture}[
Block_purp/.style={draw=black, fill = DeePONet_purple, very thick, minimum width=2cm, minimum height=0.5cm},
Blockwhite/.style={draw=black, fill = white , very thick, minimum width=0.5cm, minimum height=0.5cm},
Circ_gray_prod/.style={draw=black, fill=DeePONet_gray, shape=circle, thick, minimum width=0.7cm, minimum height=0.7cm, inner sep=0pt, outer sep=0pt},
node distance=10mm  
]
\tikzset{Circ_gray/.style={draw=black, fill=DeePONet_gray, shape=circle, thick, minimum width=0.7cm, minimum height=0.7cm, inner sep=0pt, outer sep=0pt}}

\node[draw, fill=white, very thick, minimum width=1cm, minimum height=3cm] (block_1) at (0, 0) {};

\node[Circ_gray, anchor=center] (b1) at (block_1.north) [yshift=-0.45cm] {{\footnotesize $b_1$}};  
\node[Circ_gray, anchor=center] (b2) at (block_1.north) [yshift=-1.25cm] {{\footnotesize $b_2$}};  
\node[anchor=center] (ellipsis) at (block_1.north) [yshift=-1.8cm] {{\footnotesize \vdots}};  
\node[Circ_gray, anchor=center] (bp) at (block_1.north) [yshift=-2.55cm] {{\footnotesize $b_p$}};  
    
\node[draw, fill=white, very thick, minimum width=1cm, minimum height=3cm] (block_2) at (0, -3.3cm) {};

\node[Circ_gray, anchor=center] (t1) at (block_2.north) [yshift=-0.45cm] {{\footnotesize $t_1$}};  
\node[Circ_gray, anchor=center] (t2) at (block_2.north) [yshift=-1.25cm] {{\footnotesize $t_2$}};  
\node[anchor=center] (ellipsis) at (block_2.north) [yshift=-1.8cm] {{\footnotesize \vdots}};  
\node[Circ_gray, anchor=center] (tp) at (block_2.north) [yshift=-2.55cm] {{\footnotesize $t_p$}};  

\node[Block_purp] (branch1) [left=1.0cm of block_1] {\parbox{2.0cm}{\centering Branch net}};
\node[Block_purp] (trunk1) [left=1.0cm of block_2] {\parbox{2.0cm}{\centering Trunk net}};

\node[Circ_gray_prod] (circ1) [right=of $(block_1)!0.5!(block_2)$, yshift=0mm, xshift=0.5mm, inner sep=0pt, outer sep=0pt] {$\times$};


\node[Blockwhite] (block4) [left=0.5cm of branch1] {\parbox{0.8cm}{\small \centering ${ \begin{matrix}
    u(s_1) \\
    u(s_2) \\
    \vdots \\
    u(s_m)
\end{matrix}}$}};

\node (input_trunk) at (trunk1) [shift={(-2.0cm,0.0cm)}] {$z$};

 \draw[->, thick] (block4.east) -- (branch1.west) node[midway, above] {};
 \draw[->, thick] (branch1.east) -- (b1.west) node[midway, above, xshift=-2.5mm , yshift=0.4mm] {\textcolor{blue}{$\phi_b(u)$}};
 \draw[->, thick] (branch1.east) -- (b2.west) node[midway, above] {};
 \draw[->, thick] (branch1.east) -- (bp.west) node[midway, above] {};

  \draw[->, thick] (trunk1.east) -- (t1.west) node[midway, above, xshift=-2.5mm , yshift=0.4mm] {\textcolor{blue}{$\phi_t(z)$}};
 \draw[->, thick] (trunk1.east) -- (t2.west) node[midway, above] {};
 \draw[->, thick] (trunk1.east) -- (tp.west) node[midway, above] {};
 
\draw[->, thick] (block_1.east) -- (circ1.north west) node[midway, above, xshift= 4mm, yshift= 2mm] {\textcolor{red}{$b(u)$}};
\draw[->, thick] (block_2.east) -- (circ1.south west) node[midway, above, xshift= 4mm, yshift= -5mm] {\textcolor{red}{$t(z)$}};

\draw[->, thick] (input_trunk.east) -- (trunk1.west) node[midway, above] {};

\draw[->, thick] (circ1.east) --  ++ (1cm,0cm) node[midway, above, yshift=0.3cm, xshift = 0.2cm] {{\small$G(u)(z)$}};
\end{tikzpicture}
    \vspace*{-1.25em} 
        \caption{Traditional unstacked DeepONet architecture.}
        \label{fig:DeePONet_unstacked}
    \end{subfigure}
    \caption{Standard DeepONet formulations: (a) stacked and (b) unstacked architectures.}
\end{figure*}
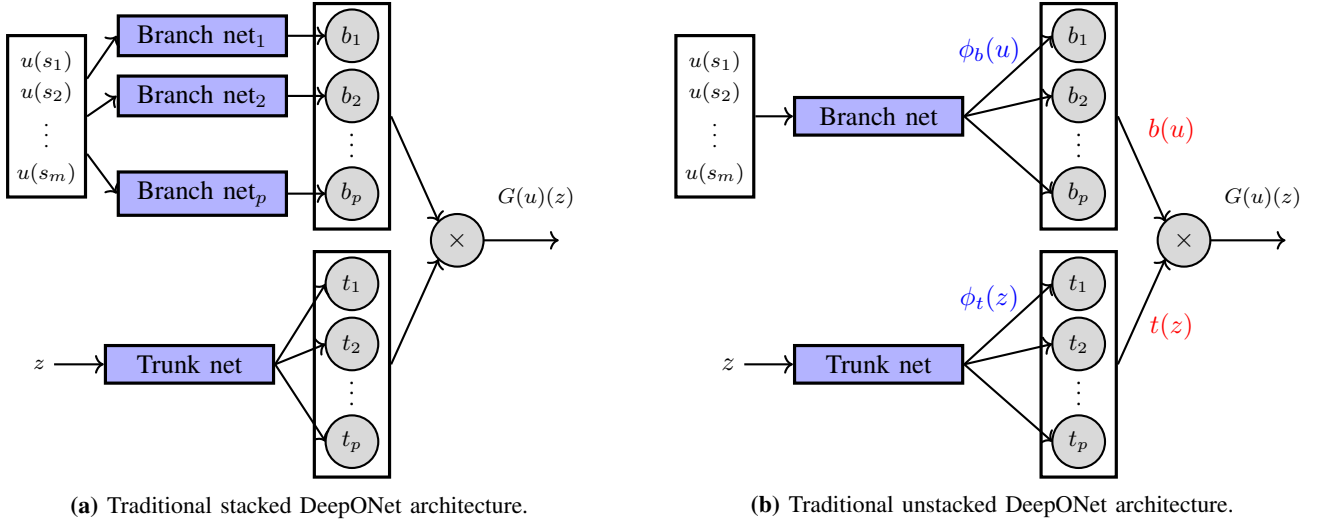
We now introduce the core equations underlying DeepONet, focusing on the MimoONet architecture illustrated in \figurename~\ref{fig:MimoONet}. This choice is motivated by the fact that MimoONet generalizes the unstacked DeepONet shown in \figurename~\ref{fig:DeePONet_unstacked}, and, as demonstrated in the Appendix, Proposition~\ref{proposition:stacked_unstacked}, the stacked DeepONet can be viewed as a special case of the unstacked version. For this reason, we present results only for the unstacked DeepONet. To simplify notation, we refer to the output equation of MimoONet as the standard DeepONet, though our results apply to all three architectures. To obtain such multi-step predictions in time using the standard DeepONet, we pass both the time and the current state $z_{j,k} := \col(x_k,jT_s)$ with $j=1,\dots,N$ as input to the trunk network. This allows us to predict $N$ future steps from the current state. Note that individual branch networks are added for each input signal, to this end define $\bu_k^i := \col(u_i(kT_s), \dots, u_i((k+N-1)T_s))$ for $i=1,\dots,n_u$. The corresponding branch network output layer takes as input the vector $\phi_b^i(\bu^i_k) \in\mathbb{R}^{n_{b,i}}$ where $n_{b,i}\in\mathbb{N}$ is the number of neurons in the last hidden layer of branch net$_i$ and has as output a vector $b^i(\bu^i_k)\in\mathbb{R}^{p n_y}$. Let $b^{i}_{q,j}(\bu^i_k)$ with $q = 1,\dots,p$, $j=1,\dots,n_y$ denote the $(q+(j-1)p)$-th element of the vector $b^i( \bu^i_k)$, then 
$$
b^i_{q,j}( \bu^i_k) = \sum_{l=1}^{n_{b,i}} w^{i,l}_{q,j} \phi_b^l(\bu^i_k) + \xi^{i}_{q,j}.
$$
The output layer of the trunk net takes in the vector $\phi_t(z_{j,k}) \in\mathbb{R}^{n_t}$ and has as output a vector of functions $t(z_{j,k})\in\mathbb{R}^{p\times 1}$.  Let $t_{i}(x_0,jT_s)$ denote the $i$-th element of the vector $t(x_0,jT_s)$, then 
$$
t_{i}(x_0,jT_s) = \sum_{l=1}^{n_t} \alpha^i_{l} \phi_t^l(x_0,jT_s) + \zeta^{i}.
$$
The product layer is then used to compute the output as:
\begin{align}\label{eqn:MimoOnet_pm}
     y_i(k+jT_s) = \sum_{l=1}^{p} \mathbf{b}^{n_u}_{l,i}(\mathbf{u}_k)
   t_l(x_0,jT_s),
\end{align}
where for simplicity of notation we define 
\begin{align}\label{eqn:mimoNet_branch}
    \mathbf{b}^{n_u}_{i,j}(\mathbf{u}_k) := b^{1}_{i,j}(\bu^1_k) b^{2}_{i,j}(\bu^2_k) \dots b^{n_u}_{i,j}(\bu^{n_u}_k).
\end{align}
The approach described above employs $n_u$ branch networks and a $1$ trunk networks in total. However, predicting multiple steps in time, that is, $t=T_s, 2T_s,\dots,NT_s$, requires evaluating the model $N$ times. Next we state a prototype nonlinear MPC problem where the standard DeepONet is employed as a prediction model. 

\begin{problem} (Prototype DeepONet MPC Problem at $t=kT_s$) \label{prob:DMPC_unstacked}
\begin{subequations}
\label{eq:3:1_1}
\begin{align}
&\min_{\bar \bu_k}  \quad   l_N(y_{N|k}) +  \sum_{i=1}^{N-1} l(y_{i|k},u_{i|k}) \label{eq:DMPCp_a1} \\ 
&\text{subject to: } \nonumber  \\
& \quad x_{0|k} =  x(kT_s), \\
& \quad y_{j|k} = \sum_{l=1}^{p} \mathbf{b}^{n_u}_{l,i}(\bar\bu_k)
   t_l(x_{0|k},jT_s), \quad j = 1,\dots, N, \label{eq:DMPCp_b1} \\
&\quad (\bar \by_k ,\bar \bu_k)  \in \Yset^{Nn_y} \times \Uset^{Nn_u}. \label{eq:DMPCp_c2}
\end{align}
\end{subequations}
\end{problem}
In the above problem, for a continuous--time signal $a(t)$, with the notation $a_{i|k}$ we denote the predicted value of $a(T_s(i+k))$ at $t=kT_s$ and the predicted state and input sequences are defined by $\bar\bu_{k} := \col(u_{0|k},\dots,u_{N-1|k})$ and $\bar \by_{k} : = \col(y_{1|k},\dots,y_{N|k})$.  In addition, $l_N(y_{N|k})$ and $l(y_{i|k},u_{i|k})$ denote the terminal cost and stage cost, respectively. Furthermore, in this approach, computing the prediction of multiple real-valued outputs requires training and using online multiple branch networks, one for each real-valued input of the underlying dynamical system. With this in mind, we present the problem statement of this paper. 
\begin{figure*}[t!]
        \centering
        \begin{tikzpicture}[
Block_purp/.style={draw=black, fill = DeePONet_purple, very thick, minimum width=1cm, minimum height=0.5cm},
Blockwhite/.style={draw=black, fill = white , very thick, minimum width=0.4cm, minimum height=0.5cm},
Circ_gray_prod/.style={draw=black, fill=DeePONet_gray, shape=circle, thick, minimum width=0.5cm, minimum height=0.5cm, inner sep=0pt, outer sep=0pt},
node distance=10mm  
]
\tikzset{
  Circ_gray/.style={
    draw=black,
    fill=DeePONet_gray,
    shape=circle,
    thick,
    minimum size=0.55cm,
    inner sep=0pt,
    outer sep=0pt,
    text width=0.55cm,
    align=center
  }
}

\tikzset{Circ_red/.style={draw=black, fill=red, shape=circle, thick, minimum width=0.5cm, minimum height=0.5cm, inner sep=0pt, outer sep=0pt}}

\tikzset{Circ_blue/.style={draw=black, fill=blue, shape=circle, thick, minimum width=0.5cm, minimum height=0.5cm, inner sep=0pt, outer sep=0pt}}

\node[Circ_red] (prod_1) at (0, 0) {$\times$};

\node[Circ_blue] (prod_2) at (1.1cm, 1.7cm) {$\times$};

\node[draw, fill=white, very thick, minimum width=0.8cm, minimum height=1.65cm, draw=red] (block_1_1) at (prod_1.west) [xshift= -1.0cm] {};

\node[Circ_gray, anchor=center] (b1_1) at (block_1_1.north) [yshift=-0.375cm] {{\footnotesize $b^1_{1,1}$}};  
\node[anchor=center] (ellipsis) at (block_1_1.north) [yshift=-0.775cm] {{\scalebox{0.6}{$\vdots$}}};  
\node[Circ_gray, anchor=center] (b1_2) at (block_1_1.north) [yshift=-1.325cm] {{\footnotesize $b^{1}_{p,1}$}};  

\node[draw, fill=white, very thick, minimum width=0.8cm, minimum height=1.65cm, draw=blue] (block_1_2) at (prod_2.west) [xshift= -2.1cm] {};

\node[Circ_gray, anchor=center] (b2_1) at (block_1_2.north) [yshift=-0.375cm] {{\footnotesize $b^1_{1,2}$}};  
\node[anchor=center] (ellipsis) at (block_1_2.north) [yshift=-0.775cm] {{\scalebox{0.6}{$\vdots$}}};  
\node[Circ_gray, anchor=center] (b2_2) at (block_1_2.north) [yshift=-1.325cm] {{\footnotesize $b^1_{p,2}$}};  

\node[draw, fill=white, very thick, minimum width=0.8cm, minimum height=1.65cm, draw=red] (block_2_1) at (prod_1.east) [xshift= 2.2cm] {};

\node[Circ_gray, anchor=center] (b2_1) at (block_2_1.north) [yshift=-0.375cm] {{\footnotesize $b^2_{1,1}$}};  
\node[anchor=center] (ellipsis) at (block_2_1.north) [yshift=-0.775cm] {{\scalebox{0.6}{$\vdots$}}};  
\node[Circ_gray, anchor=center] (b2_2) at (block_2_1.north) [yshift=-1.325cm] {{\footnotesize $b^2_{p,1}$}};  

\node[draw, fill=white, very thick, minimum width=0.8cm, minimum height=1.65cm, draw=blue] (block_2_2) at (prod_2.east) [xshift= 1.1cm] {};

\node[Circ_gray, anchor=center] (b2_1) at (block_2_2.north) [yshift=-0.375cm] {{\footnotesize $b^2_{1,2}$}};  
\node[anchor=center] (ellipsis) at (block_2_2.north) [yshift=-0.775cm] {{\scalebox{0.6}{$\vdots$}}};  
\node[Circ_gray, anchor=center] (b2_2) at (block_2_2.north) [yshift=-1.325cm] {{\footnotesize $b^2_{p,2}$}};  

\node[draw, fill=none, very thick, minimum width=0.9cm, minimum height= 3.45cm, draw=black] (block_u2_black) at (block_2_2.south) [xshift= 0cm, yshift= -0.005cm] {};

\node[draw, fill=white, very thick, minimum width=2.35cm, minimum height= 0.65cm, draw=black] (block_3_1) at (prod_1.north) [xshift= 0.5cm, yshift= 3.0cm] {};

\node[Circ_gray, anchor=center] (t1_1) at (block_3_1.west) [xshift=0.35cm] {{\footnotesize $t_1$}};  
\node[Circ_gray, anchor=center] (t1_2) at (block_3_1.west) [xshift=1.05cm] {{\footnotesize $t_2$}};  
\node[anchor=center] (ellipsis) at (block_3_1.west) [xshift=1.5cm] {{\tiny \dots}};  
\node[Circ_gray, anchor=center] (t1_p) at (block_3_1.west) [xshift=1.95cm] {{\footnotesize $t_p$}};  

\node[draw, fill=none, very thick, minimum width=0.9cm, minimum height= 3.45cm, draw=black] (block_u1_black) at (block_1_2.south) [xshift= 0cm, yshift= -0.005cm] {};

\node[Block_purp] (branch1) [left=0.75cm of block_u1_black] {\parbox{1.8cm}{\centering {\small Branch net$_1$}}};

\node[Block_purp] (branch2) [right=0.75cm of block_u2_black] {\parbox{1.8cm}{\centering {\small Branch net$_2$}}};



\node[Block_purp] (trunk1) [left=1.75cm of block_3_1] {\parbox{1.8cm}{\centering {\small Trunk net}}};

\node (input_trunk) [left=0.75cm of trunk1] { \footnotesize \centering ${\hspace{-0.4mm} \begin{matrix}
    x(0) \\
    jT_s
\end{matrix}}$};

\node[Blockwhite] (input_1) [left=0.75cm of branch1] {\parbox{1.9cm}{ \footnotesize \centering ${\hspace{-0.4mm} \begin{matrix}
    u_1(0) \\
    u_1(T_s) \\
    \vdots \\
    u_1((N-1)T_s)
\end{matrix}}$}};

\node[Blockwhite] (input_2) [right=0.75cm of branch2] {\parbox{1.9cm}{ \footnotesize \centering ${\hspace{-0.4mm} \begin{matrix}
    u_2(0) \\
    u_2(T_s) \\
    \vdots \\
    u_2((N-1)T_s)
\end{matrix}}$}};

\node (output_1) at (prod_1) [shift={(0.0cm,-1.0cm)}] {\small \textcolor{red}{$ y_1(jT_s)$}};
\node (output_2) at (prod_2) [shift={(0.0cm,-2.7cm)}] {\small \textcolor{blue}{$ y_2(jT_s)$}};

 \draw[->, thick] (input_1.east) -- (branch1.west) node[midway, above] {};
 \draw[->, thick] (input_2.west) -- (branch2.east) node[midway, above] {};
 \draw[->, thick] (input_trunk.east) -- (trunk1.west) node[midway, above] {};

 \draw[->, thick] (trunk1.east) -- (block_3_1.west) node[midway, above] {};
 \draw[->, thick] (branch1.east) -- (block_u1_black.west) node[midway, above] {};
 \draw[->, thick] (branch2.west) -- (block_u2_black.east) node[midway, above] {};

 \draw[->, thick, red] (block_1_1.east) -- (prod_1.west) node[midway, above] {};
 \draw[->, thick, blue] (block_1_2.east) -- (prod_2.west) node[midway, above] {};

 \draw[->, thick, red] (block_2_1.west) -- (prod_1.east) node[midway, above] {};
 \draw[->, thick, blue] (block_2_2.west) -- (prod_2.east) node[midway, above] {};

\draw[->, thick, red] (prod_1.north  |- block_3_1.south) -- (prod_1.north);
\draw[->, thick, blue] (prod_2.north  |- block_3_1.south) -- (prod_2.north);

 \draw[->, thick, red] (prod_1.south) -- (output_1.north) node[midway, above] {};
 \draw[->, thick, blue] (prod_2.south) -- (output_2.north) node[midway, above] {};
\end{tikzpicture}
        \caption{Illustration of the stacked MimoONet \cite{MimoONet} for 2 inputs and 2 outputs.}
        \label{fig:MimoONet}
\end{figure*}
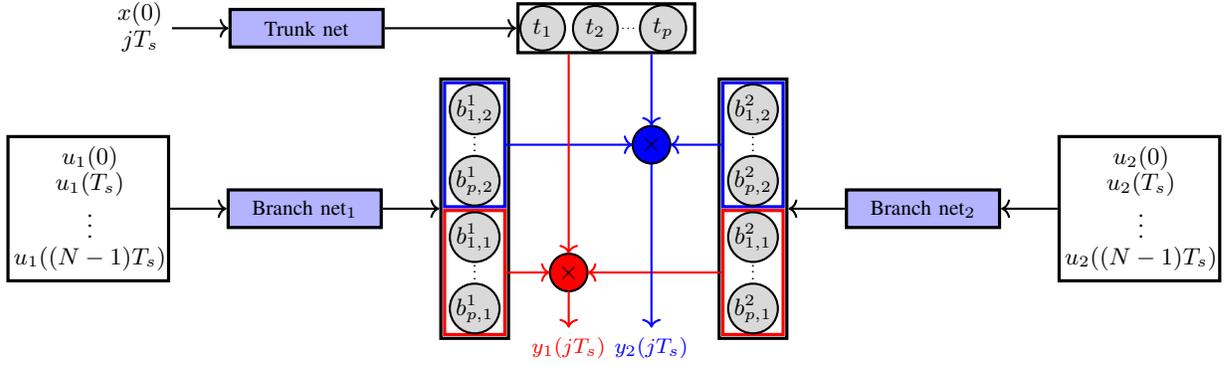  

\paragraph*{Problem statement}
Given the nonlinear MIMO system in~\eqref{eqn:systems}, the goal is to approximate the multi-step operator mapping the initial state \( x_k \) and input \( \bu_k \) to future outputs \( \by_k \) using the DeepONet framework. The model should predict multiple time steps across multiple outputs in a single network evaluation, enabling efficient predictive control. Crucially, the approach must also satisfy the universal approximation theorem to ensure representational completeness.

\section{THE MULTI-STEP DEEPONET ARCHITECTURE}
\label{sec3}
Since we are only interested in the solution $y(t)$ of system \eqref{eqn:systems} at discrete time instances  $t=(k+1)T_s, (k+2)T_s,\dots,(k+N)T_s$, given a piecewise-constant input signal $\{u(t)\}_{t\in[kT_s,(k+N-1)T_s)}$ and initial condition $x_k:= x(kT_s)$ at time $t=kT_s$, the multi-step output solution $\mathbf{y}_k$ for the continuous-time dynamical system \eqref{eqn:systems} is given by
\begin{align}
\label{eqn:operator}
  \by_k &=\mathcal{G}(\bu_k)(x_k) : =\nonumber\\  &\begin{bmatrix}
     h\left(x_k + \sum_{i=0}^0 \int_{(k+i) T_s}^{(k+i+1) T_s} f(x(\tau), [\mathbf{u}_k]^{(i+1) n_u}_{i n_u+1} ) \, d\tau\right) \\
     h\left(x_k + \sum_{i=0}^1\int_{(k+i) T_s}^{(k+i+1) T_s} f(x(\tau), [\mathbf{u}_k]^{(i+1) n_u}_{i n_u+1}) \, d\tau\right) \\
    \vdots \\
    h\left(x_k + \sum_{i=0}^{N-1}\int_{(k+i) T_s}^{(k+i+1) T_s} f(x(\tau), [\mathbf{u}_k]^{(i+1) n_u}_{i n_u+1} ) \, d\tau\right)
\end{bmatrix},
\end{align}
where $\mathbf{y}_k = \col(\mathcal{G}^1(\mathbf{u}_k)(x_k),\dots,\mathcal{G}^{Nn_y}(\mathbf{u}_k)(x_k))\in\mathbb{R}^{Nn_y}$ is a vector of $Nn_y$ operators mapping real vectors in $\mathbb{U}^{N}\subset \mathbb{R}^{Nn_u}$ to continuous functions in $C(\mathbb{X})$. Notice that time is implicitly encoded and there is no longer a need to use the time index $jT_s$ as an input to the trunk network. Next, we present the universal approximation theorem from \cite{Chen_ONN}. 

\begin{theorem}[Theorem 5, \citen{Chen_ONN}]\label{theorem:universal_approx}
 Suppose that $\sigma\in(TW)$, $X$ is a Banach Space, $K_1 \subset X$, $K_2 \subset \mathbb{R}^d$ are two compact sets in $X$ and $\mathbb{R}^d$, and $V$ is a compact set in $C(K_1)$, $G$ is a nonlinear continuous operator, which maps $V$ into $C(K_2)$. Then for any $\epsilon > 0$, there are positive integers $n$, $p$, $m$, constants $c_i^k$, $\xi_{ij}^k$, $\theta_i^k$, $\zeta_k\in\mathbb{R}$. $w_k\in\mathbb{R}^d$, $s_j\in K_1$, $i = 1,\dots,n$, $k=1,\dots,p$, $j=1,\dots,m$, such that 
\begin{equation}\label{eqn:universal_approx_trm}
    \begin{split}
    \bigg| &G(u)(z) - \\&\sum_{k=1}^p\sum_{i=1}^n c_i^k \sigma \left( \sum_{j=1}^m \xi_{ij}^k u(s_j) + \theta_i^k\right)\sigma\left( w_k \cdot z + \zeta_k \right) \bigg| < \epsilon
\end{split}
\end{equation}
holds for all $u\in V$ and $z\in K_2$. 
\end{theorem}
The classical universal approximation theorem holds for operators that map continuous functions $u\in V$ to continuous functions $C(K_2)$. This means that the universal approximation theorem is not directly applicable operators $\mathcal{G}$ of the form \eqref{eqn:operator}. Next, we present a slightly modified universal approximation theorem.
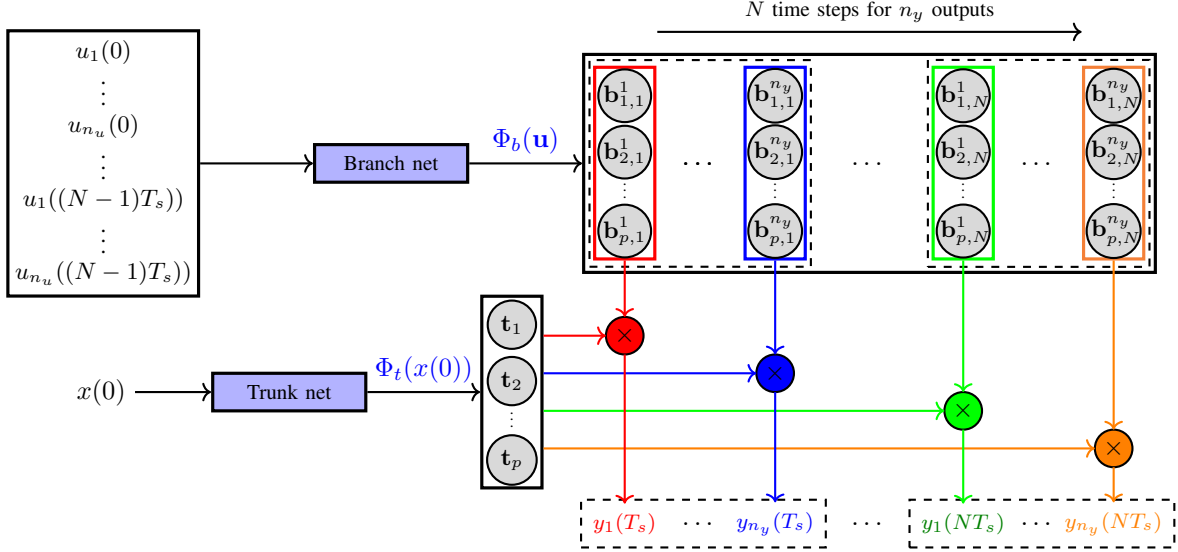
\begin{figure*}[t!]
        \centering
        \begin{tikzpicture}[
Block_purp/.style={draw=black, fill = DeePONet_purple, very thick, minimum width=1cm, minimum height=0.5cm},
Blockwhite/.style={draw=black, fill = white , very thick, minimum width=0.4cm, minimum height=0.5cm},
Circ_gray_prod/.style={draw=black, fill=DeePONet_gray, shape=circle, thick, minimum width=0.5cm, minimum height=0.5cm, inner sep=0pt, outer sep=0pt},
node distance=10mm  
]
\tikzset{
  Circ_gray/.style={
    draw=black,
    fill=DeePONet_gray,
    shape=circle,
    thick,
    minimum size=0.60cm,
    inner sep=0pt,
    outer sep=0pt,
    text width=0.60cm,
    align=center
  }
}

\tikzset{Circ_red/.style={draw=black, fill=red, shape=circle, thick, minimum width=0.5cm, minimum height=0.5cm, inner sep=0pt, outer sep=0pt}}

\tikzset{Circ_blue/.style={draw=black, fill=blue, shape=circle, thick, minimum width=0.5cm, minimum height=0.5cm, inner sep=0pt, outer sep=0pt}}

\tikzset{Circ_green/.style={draw=black, fill=green, shape=circle, thick, minimum width=0.5cm, minimum height=0.5cm, inner sep=0pt, outer sep=0pt}}

\tikzset{Circ_orange/.style={draw=black, fill=orange, shape=circle, thick, minimum width=0.5cm, minimum height=0.5cm, inner sep=0pt, outer sep=0pt}}

\coordinate (prod_1) at (0.0cm, 0.0cm);

\node[draw, fill=white, very thick, minimum width=0.8cm, minimum height=2.55cm, draw=red] (block_red) at (prod_1.east) [yshift= -0.5cm, xshift= -6.5cm] {};

\node[Circ_gray, anchor=center] (b1_1) at (block_red.north) [yshift=-0.4cm] {{\footnotesize $\mathbf{b}^1_{1,1}$}};  
\node[Circ_gray, anchor=center] (b1_2) at (block_red.north) [yshift=-1.15cm] {{\footnotesize $\mathbf{b}^{1}_{2,1}$}};  
\node[anchor=center] (ellipsis) at (block_red.north) [yshift=-1.6cm] {{\scalebox{0.6}{$\vdots$}}};  
\node[Circ_gray, anchor=center] (b1_3) at (block_red.north) [yshift=-2.2cm] {{\footnotesize $\mathbf{b}^{1}_{p,1}$}}; 

\node[draw, fill=white, very thick, minimum width=0.8cm, minimum height=2.55cm, draw=blue] (block_blue) at (prod_1.east) [yshift= -0.5cm, xshift= -4.5cm] {};

\node[Circ_gray, anchor=center] (b2_1) at (block_blue.north) [yshift=-0.4cm] {{\footnotesize $\mathbf{b}^{n_y}_{1,1}$}};  
\node[Circ_gray, anchor=center] (b2_2) at (block_blue.north) [yshift=-1.15cm] {{\footnotesize $\mathbf{b}^{n_y}_{2,1}$}};  
\node[anchor=center] (ellipsis) at (block_blue.north) [yshift=-1.6cm] {{\scalebox{0.6}{$\vdots$}}};  
\node[Circ_gray, anchor=center] (b2_3) at (block_blue.north) [yshift=-2.2cm] {{\footnotesize $\mathbf{b}^{n_y}_{p,1}$}}; 

\node[draw, black, thick, dashed, minimum width=2.4cm, minimum height=2.7cm, fit=(b1_1)(b1_2)(ellipsis)(b1_3)(b2_1)(b2_2)(b2_3)] (block_group_`1) {};

\node[anchor=center] (dots) at ($(block_red.east)!0.5!(block_blue.west)$) {\scalebox{1.0}{$\dots$}};  

\node[draw, fill=white, very thick, minimum width=0.8cm, minimum height=2.55cm, draw=green] (block_green) at (prod_1.east) [yshift= -0.5cm, xshift= -2.0cm] {};
\node[Circ_gray, anchor=center] (b3_1) at (block_green.north) [yshift=-0.4cm] {{\footnotesize $\mathbf{b}^1_{1,N}$}};  
\node[Circ_gray, anchor=center] (b3_2) at (block_green.north) [yshift=-1.15cm] {{\footnotesize $\mathbf{b}^{1}_{2,N}$}};  
\node[anchor=center] (ellipsis) at (block_green.north) [yshift=-1.6cm] {{\scalebox{0.6}{$\vdots$}}};  
\node[Circ_gray, anchor=center] (b3_3) at (block_green.north) [yshift=-2.2cm] {{\footnotesize $\mathbf{b}^{1}_{p,N}$}}; 

\node[draw, fill=white, very thick, minimum width=0.8cm, minimum height=2.55cm, draw=Orange] (block_orange) at (prod_1.east) [yshift= -0.5cm, xshift= 0.0cm] {};

\node[Circ_gray, anchor=center] (b4_1) at (block_orange.north) [yshift=-0.4cm] {{\footnotesize $\mathbf{b}^{n_y}_{1,N}$}};  
\node[Circ_gray, anchor=center] (b4_2) at (block_orange.north) [yshift=-1.15cm] {{\footnotesize $\mathbf{b}^{n_y}_{2,N}$}};  
\node[anchor=center] (ellipsis) at (block_orange.north) [yshift=-1.6cm] {{\scalebox{0.6}{$\vdots$}}};  
\node[Circ_gray, anchor=center] (b4_3) at (block_orange.north) [yshift=-2.2cm] {{\footnotesize $\mathbf{b}^{n_y}_{p,N}$}}; 

\node[draw, black, thick, dashed, minimum width=2.4cm, minimum height=2.7cm, fit=(b3_1)(b3_2)(ellipsis)(b4_3)(b4_1)(b4_2)(b4_3)] (block_group_2) {};

\node[anchor=center] (dots) at ($(block_red.east)!0.5!(block_blue.west)$) {\scalebox{1.0}{$\dots$}};  

\node[anchor=center] (dots) at ($(block_green.east)!0.5!(block_orange.west)$) {\scalebox{1.0}{$\dots$}};  

\node[anchor=center] (dots) at ($(block_blue.east)!0.5!(block_green.west)$) {\scalebox{1.0}{$\dots$}};  

\node[
  draw=black,
  very thick,
  minimum width=7.6cm,
  minimum height=2.9cm,
  fit=(b1_1)(b1_2)(b1_3)(b2_1)(b2_2)(b2_3)
       (b3_1)(b3_2)(b3_3)(b4_1)(b4_2)(b4_3),
  label={[yshift=0.1cm]above:{}}
] (big_group) {};

\draw[->, thick] 
  ([xshift=+1.0cm, yshift=1.75cm]big_group.west) 
    -- 
  ([xshift=-1.0cm, yshift=1.75cm]big_group.east) 
  node[midway, above] {\small $N$ time steps for $n_y$ outputs};

\node[Circ_red] (prod_red) at ($(block_red.south)+(0,-1.0cm)$) {$\times$};
\node[Circ_blue] (prod_blue) at ($(block_blue.south)+(0,-1.5cm)$) {$\times$};

\node[Circ_green] (prod_green) at ($(block_green.south)+(0,-2.0cm)$) {$\times$};
\node[Circ_orange] (prod_orange) at ($(block_orange.south)+(0,-2.5cm)$) {$\times$};

\node[draw, fill=white, very thick, minimum width=0.8cm, minimum height= 2.55cm, draw=black] (trunk_output) at (prod_orange.south) [xshift= -8.0cm, yshift= 1.0cm] {};

\node[Circ_gray, anchor=center] (b1_1) at (trunk_output.north) [yshift=-0.4cm] {{\footnotesize $\mathbf{t}_{1}$}};  
\node[Circ_gray, anchor=center] (b1_2) at (trunk_output.north) [yshift=-1.15cm] {{\footnotesize $\mathbf{t}_{2}$}};  
\node[anchor=center] (ellipsis) at (trunk_output.north) [yshift=-1.6cm] {{\scalebox{0.6}{$\vdots$}}};  
\node[Circ_gray, anchor=center] (b1_3) at (trunk_output.north) [yshift=-2.2cm] {{\footnotesize $\mathbf{t}_{p}$}};

\draw[->, thick, red] (prod_red.west  -| trunk_output.east) -- (prod_red.west)
  node[midway, above, xshift=0.0cm] {};

\draw[->, thick, blue] (prod_blue.west  -| trunk_output.east) -- (prod_blue.west)
  node[midway, above, xshift=0.5cm] {};

\draw[->, thick, green] (prod_green.west  -| trunk_output.east) -- (prod_green.west)
  node[midway, above, xshift=0.5cm] {};

\draw[->, thick, orange] (prod_orange.west  -| trunk_output.east) -- (prod_orange.west)
  node[midway, above, xshift=0.5cm] {};

\draw[->, thick, red] (block_red.south) -- (prod_red.north);
\draw[->, thick, blue] (block_blue.south) -- (prod_blue.north);
\draw[->, thick, green] (block_green.south) -- (prod_green.north);
\draw[->, thick, orange] (block_orange.south) -- (prod_orange.north);

\node (output_1) at (prod_red) [shift={(0.0cm,-2.5cm)}] {\small \textcolor{red}{$y_1(T_s)$}};
\node (output_2) at (prod_blue) [shift={(0.0cm,-2.0cm)}] {\small \textcolor{blue}{$y_{n_y}(T_s)$}};
\node (output_3) at (prod_green) [shift={(0.0cm,-1.5cm)}] {\small \textcolor{darkgreen}{$y_1(NT_s)$}};
\node (output_4) at (prod_orange) [shift={(0.0cm,-1.0cm)}] {\small \textcolor{orange}{$y_{n_y}(NT_s)$}};

\node (dots) at ($(prod_red)!0.5!(prod_blue) + (0,-2.25cm)$) {$\dots$};
\node (dots) at ($(prod_green)!0.5!(prod_orange) + (0,-1.25cm)$) {$\dots$};
\node (dots) at ($(prod_blue)!0.5!(prod_green) + (0,-1.75cm)$) {$\dots$};

\node[draw=black, thick, dashed, minimum width=2.3cm, minimum height=0.50cm, fit=(output_1)(output_2), inner sep=1pt] (box_outputs) {};

\node[draw=black, thick, dashed, minimum width=2.3cm, minimum height=0.50cm, fit=(output_3)(output_4), inner sep=1pt] (box_outputs) {};

\draw[->, thick, red] (prod_red.south) -- (output_1.north);
\draw[->, thick, blue] (prod_blue.south) -- (output_2.north);
\draw[->, thick, green] (prod_green.south) -- (output_3.north);
\draw[->, thick, orange] (prod_orange.south) -- (output_4.north);

\node[Block_purp] (trunk_net) [left=1.5cm of trunk_output] {\small \parbox{1.8cm}{\centering {Trunk net}}};

\node[Block_purp] (Branch_net) [left=1.5cm of big_group] {\small \parbox{1.8cm}{\centering {Branch net}}};

\node[Blockwhite] (input_1) [left=1.5cm of Branch_net] {\footnotesize \parbox{2.3cm}{\centering ${\hspace{-0.2mm} \begin{matrix}
    u_1(0) \\
    \vdots \\
    u_{n_u}(0) \\
    \vdots \\
    u_1((N-1)T_s)) \\
    \vdots \\
    u_{n_u}((N-1)T_s))
\end{matrix}}$}};

\node (input_trunk) at (trunk_net) [shift={(-2.5cm,0.0cm)}] {$x(0)$};

\draw[->, thick] (input_trunk.east) -- (trunk_net.west);
\draw[->, thick] (trunk_net.east) -- node[above]{\textcolor{blue}{$\Phi_t(x(0))$}}(trunk_output.west);
\draw[->, thick] (input_1.east) -- (Branch_net.west);
\draw[->, thick] (Branch_net.east) -- node[above]{\textcolor{blue}{$\Phi_b(\mathbf{u})$}}(big_group.west);

\end{tikzpicture}
        \caption{Multi-step DeepONet architecture: the branch input at time instant $t = 0$ is the multi-step input sequence $\bar\bu_0$; the trunk input is the measured state $x(0)$; the output is the predicted multi-step output sequence $\bar\bu_0$.}
        \label{fig:DeePONet_multi_step}
\end{figure*}

\begin{theorem}\label{theorem:universal_approx_2}
Suppose that $\sigma \in (TW)$, $K_1 \subset \mathbb{R}^{l}$ is a compact set in $\mathbb{R}^{l}$ and $K_2 \subset \mathbb{R}^d$ is a compact set in $\mathbb{R}^d$, $G$ is a nonlinear continuous operator, which maps a $K_1$ into $C(K_2)$. Then for any $\epsilon > 0$, there are positive integers $n$, $p$, constants $c_i^k$, $\xi_{i}^k\in\mathbb{R}^l$, $\theta_i^k$, $\zeta_k\in\mathbb{R}$ $w_k\in\mathbb{R}^d$, $i = 1,\dots,n$, $k=1,\dots,p$, such that 
\begin{equation}\label{eqn:universal_approx_trm1}
    \begin{split}
    \left| G(\mathbf{u})(x) -\sum_{k=1}^p\sum_{i=1}^n c_i^k \sigma \left( \xi_{i}^k \cdot \mathbf{u} + \theta_i^k\right)\sigma\left( w_k \cdot x + \zeta_k \right) \right| < \epsilon
\end{split}
\end{equation}
holds for all $\mathbf{u} \in K_1 $ and $x \in K_2$.
\end{theorem}
\begin{proof} From the assumption that $G$ is a continuous operator which maps a compact set $K_1$ of $\mathbb{R}^l$ into $C(K_2)$, it is straightforward to prove that the range $G(K_1) = \{G(u):u\in K_1\}$ is also a compact set in $C(K_2)$. By Theorem~\ref{theorem_2} given in the Appendix, for any $\epsilon>0$, there are real numbers $c_k(G(u))$ and $\zeta_k$, vectors $\omega_k\in\mathbb{R}^d$, $k=1,\dots,p$, such that 
    \begin{equation}\label{eqn:theorem_3_eq1}
       \left| G(\mathbf{u})(x)-\sum_{k=1}^p c_k(G(\mathbf{u}))\sigma(\omega_k\cdot x + \zeta_k) \right| < \epsilon/2,
    \end{equation}
    holds for all $\mathbf{u} \in K_1 $ and $x \in K_2$. Note that \( G: K_1 \rightarrow C(K_2) \) is a continuous mapping and \( c_k : C(K_2) \rightarrow \mathbb{R} \) is a linear continuous functional. Hence, the composition \( c_k \circ G : K_1 \rightarrow \mathbb{R} \) is a continuous function on \( K_1 \) (Proposition~8.4, \cite[s.~84]{sutherland2009introduction}). This shows that $c_k(G(\mathbf{u}))\in C(K_1)$ and we can apply Theorem~\ref{theorem_2} again to show that 
    for any $\epsilon>0$, there are a positive integer $p_k$, real numbers $c^k_i$, $\theta^k_i$ and vectors $\xi^k_i\in\mathbb{R}^l$, where $k=1,\dots,p$, , $i=1,\dots,n_k$, such that 
    \begin{equation}\label{eqn:theorem_3_eq2}
       \left|   c_k(G(\mathbf{u}))-\sum_{i=1}^{n_k} c^k_i \sigma(\xi^k_i \cdot \mathbf{u} + \theta^k_i) \right|  < \frac{\epsilon}{2 L},
    \end{equation}
    holds for all $k=1,\dots,p$ and $\mathbf{u}\in K_1$, where
    $$
    L = \sum_{k=1}^p\sup_{x \in K_2}\left| \sigma(\omega_k\cdot x + \zeta_k))\right|.
    $$
Substitute \eqref{eqn:theorem_3_eq2} into \eqref{eqn:theorem_3_eq1}, and let $n = \max_{k}n_k$  and $c_{k}^i = 0$ for all $n_k<i\leq n$, then we obtain
\begin{equation}\label{eqn:universal_approx_trm2}
    \begin{split}
    \bigg| G(\mathbf{u})(x) -\sum_{k=1}^p\sum_{i=1}^n c_i^k \sigma \left( \xi_{i}^k \cdot \mathbf{u} + \theta_i^k\right)\sigma\left( w_k \cdot x + \zeta_k \right) \bigg| < \epsilon
\end{split}
\end{equation}
for all $\mathbf{u} \in K_1 $ and $x \in K_2$, which finalizes the proof.
\end{proof}
Note that this removes the need for multiple branch networks to encode each input signal individually, as for the standard DeepONet. With this in mind, we present a new DeepONet structure in the next subsection. 
\subsection{MS-DeepONet architecture}\label{subsec:MS_architecture}
In this section we present a novel Multi-Step DeepONet (MS-DeepONet) architecture. An adaptation compared to the standard DeepONet is that a single branch network is used to encode the complete sequence of inputs $\bu_k$ at once. The second adaption is the addition of $p$ neurons in the output layer for each output of the system and each prediction step in time, leading to a total of $pn_yN$ linear activation functions in the output layer of the branch network to implicitly encode $N$ time steps; see \figurename~\ref{fig:DeePONet_multi_step}. This is different from the standard DeepONet where each branch network has $pn_y$ linear activation functions in each output layer of the $n_u$ branch networks and time is given as input to the trunk network. The output layer of the branch net of the MS-DeepONet takes in the vector $\Phi_b(\mathbf{u}_k) := \col(\phi_b^1(\mathbf{u}_k),\dots,\phi_b^{n_b}(\mathbf{u}_k))$ and has linear neurons in the output layer: $\mathbf{b}^q_{i,j}(\mathbf{u}_k)$ with $i=1,\dots,p$, $j=1,\dots,N$ and $q = 1,\dots,n_y$ given by 
\begin{align}\label{eqn:branch_def}
    \mathbf{b}^q_{i,j}(\mathbf{u}_k) = \sum_{l=1}^{n_b} w_{i,j}^{q,l} \phi^l_b(\mathbf{u}_k) + \xi_{i,j}^q,
\end{align}
with weights $w_{i,j}^{q,l}\in \mathbb{R}$ and biases $\xi_{i,j}^q\in\mathbb{R}$. The output layer of the trunk net takes in the vector $\Phi_t(x_k) := \col(\phi_t^1(x_k),\dots,\phi_t^{n_t}(x_k))$ and has linear activation functions in the output layer: $\mathbf{t}_{i}(x(k))$ with $i=1,\dots,p$ given by 
\begin{align}\label{eqn:trunk_def}
\mathbf{t}_{i}(x_k) = \sum_{k=1}^{n_t} 
\alpha^k_{i} \phi_t^k(x_k) + \zeta_{i}.
\end{align}
The product layer is then used to compute the output as follows:
\begin{equation}\label{eqn:output_MS_DeePONet}
    [\mathbf{y}_k]_{q+(j-1)n_y} = \sum_{i=1}^{p} \mathbf{b}^q_{i,j}(\mathbf{u}_k) \mathbf{t}_i(x_k)
\end{equation}
where the same indices are used as in \figurename~\ref{fig:DeePONet_multi_step}. 
\begin{remark}
    From Proposition~\ref{proposition:stacked_unstacked} in the Appendix and Theorem~\ref{theorem:universal_approx_2} it follows that for any $\epsilon>0$ there exist a branch network and trunk network as defined in \eqref{eqn:branch_def} and \eqref{eqn:trunk_def} respectively, such that
\begin{equation}\label{eqn:universal_approx_trm3}
    \begin{split}
    \bigg| \mathcal{G}^{q+(j-1)n_y}(\mathbf{u}_k)(x_k) - \sum_{i=1}^{p} \mathbf{b}^q_{i,j}(\mathbf{u}_k) \mathbf{t}_i(x_k) \bigg| < \epsilon.
\end{split}
\end{equation}
This implies that the MS-DeepONet is a universal approximation for MIMO operators $\mathcal{G}$ of the from \eqref{eqn:operator}.
\end{remark}
For convenience of notation, let
\begin{align}\label{eqn:trunk_matrix}
    \mathbf{t}(x_k):= \col(\mathbf{t}_1(x_k),\dots,\mathbf{t}_p(x_k)),
\end{align}
and let
\begin{equation}\label{eqn:branch_matrix}
    \begin{aligned}
    &\mathbf{B}(\mathbf{u}_k) := \\
    &\begin{bmatrix}
        \mathbf{b}^1_{1,1}(\mathbf{u}_k)  \dots  \mathbf{b}^{n_y}_{1,1}(\mathbf{u}_k) & \dots &  \mathbf{b}^1_{1,N}(\mathbf{u}_k)  \dots  \mathbf{b}^{n_y}_{1,1}(\mathbf{u}_k) \\
        \mathbf{b}^1_{2,1}(\mathbf{u}_k)  \dots  \mathbf{b}^{n_y}_{2,1}(\mathbf{u}_k) & \dots &  \mathbf{b}^1_{2,N}(\mathbf{u}_k)  \dots  \mathbf{b}^{n_y}_{2,1}(\mathbf{u}_k) \\
        \vdots &  \dots & \vdots \\
        \mathbf{b}^1_{p,1}(\mathbf{u}_k)  \dots  \mathbf{b}^{n_y}_{p,1}(\mathbf{u}_k) & \dots &  \mathbf{b}^1_{p,N}(\mathbf{u}_k)  \dots  \mathbf{b}^{n_y}_{p,1}(\mathbf{u}_k) 
    \end{bmatrix},
\end{aligned}
\end{equation}
then the MS-DeepONet can be expressed as a single matrix vector equation, i.e.,
\begin{align}\label{eqn:DeepONet_output}
    \mathbf{y}_k := \mathbf{B}^\top(\mathbf{u}_k)\mathbf{t}(x_k).
\end{align} 
In the next subsection, we derive the underlying basis of the proposed MS-DeepONet, providing new insights in the underlying structure of DeepONet. 

\subsection{The underlying structure of the MS-DeepONet architecture: an addaptive basis viewpoint}\label{subsec:basis}
This subsection examines the structure of MS-DeepONet, highlighting its computational advantages over traditional feedforward networks in predictive control. It also enables developing direct data-driven predictive control schemes \cite{Basis_Lazar} and alternative training and initialization methods \cite{Adaptive_basis_Cyr}. In order to extract a candidate basis, for the MS-DeepONet output equation \eqref{eqn:DeepONet_output}, define the bias coefficient vector and weight matrix for the branch net 
\begin{align}\label{eqn:branch_coef}
    \boldsymbol{\xi}^b_{r} = \begin{bmatrix}
        \xi_{1,j}^q \\
        \xi_{2,j}^q \\
        \vdots \\
        \xi_{p,j}^q 
    \end{bmatrix}, \quad  \mathbf{W}^b_{r} = \begin{bmatrix}
        w_{1,j}^{q,1} &  w_{1,j}^{q,2} & \dots & w_{1,j}^{q,n_b} \\
       w_{2,j}^{q,1} & w_{2,j}^{q,2} & \dots & w_{2,j}^{q,n_b} \\
        \vdots \\
         w_{p,j}^{q,1} & w_{p,j}^{q,2} & \dots & w_{p,j}^{q,n_b} \\
    \end{bmatrix},
\end{align}
with $r = q+(j-1)n_y$ and $q = 1,\dots,n_y$, $j = 1, \dots, N$. Similarly, define the bias coefficient vector and weight matrix for the trunk net
\begin{align}\label{eqn:trunk_coef}
    \boldsymbol{\zeta}^t = \begin{bmatrix}
        \zeta_1 \\
        \zeta_2 \\
        \vdots \\
        \zeta_p 
    \end{bmatrix}, \quad  \mathbf{W}^t = \begin{bmatrix}
        \alpha_{1}^1 &  \alpha_{1}^2 & \dots & \alpha_{1}^{n_t} \\
        \alpha_{2}^1 &  \alpha_{2}^2 & \dots & \alpha_{2}^{n_t} \\
        \vdots \\
        \alpha_{p}^1 &  \alpha_{p}^2 & \dots & \alpha_{p}^{n_t}\\
    \end{bmatrix}.
\end{align}
The MS-DeepONet output equation \eqref{eqn:output_MS_DeePONet} can be written component-wise as:
\begin{align}\label{eqn:multi_step_output}
    [\mathbf{y}_k]_i = \left( \mathbf{W}^b_i \Phi_b(\mathbf{u}_k)+\boldsymbol{\xi}^b_i \right)^\top \left( \mathbf{W}^t \Phi_t(x_k)+\boldsymbol{\zeta}^t \right),
\end{align}
with $i = 1,\dots,Nn_y$. In the following Lemma we derive a candidate basis for the MS-DeepONet. 
\begin{lemma}\label{lemma:basis}
    There exist a coefficient matrix $\Theta_o$ and a vector $\Phi_{\otimes}(\mathbf{u}_k,\mathbf{x}_0)$ such that 
    \begin{align}\label{eqn:output_basis}
        \mathbf{y}_k = \Theta_o\Phi_{\otimes}(\mathbf{u}_k,x_k):=\Theta_o \begin{bmatrix}
        \Phi_b(\mathbf{u}_k)\otimes \Phi_t(x_k) \\
        \Phi_b(\mathbf{u}_k) \\
        \Phi_t(x_k) \\
        1
    \end{bmatrix},
    \end{align}
    where the coefficient matrix contains all the weights and biases from the branch and trunk output layers, i.e., 
    \begin{equation}\label{eqn:basis}
    \begin{aligned}
       &\Theta_o = \\
       &\begin{bmatrix}
    \text{vec}^{\top}(\mathbf{W}_1^{b^\top} \mathbf{W}^t) & \boldsymbol{\zeta}^{t^\top}\mathbf{W}^b_1 & \boldsymbol{\xi}_1^{b^\top}\mathbf{W}^t & \boldsymbol{\xi}_1^{b^\top} \boldsymbol{\zeta}^t \\ 
    \vdots  & \vdots & \vdots & \vdots \\
    \text{vec}^{\top}(\mathbf{W}_{n_yN}^{b^\top} \mathbf{W}^t) & \boldsymbol{\zeta}^{t^\top} \mathbf{W}^b_{n_yN} & \boldsymbol{\xi}_{n_yN}^{b^\top}\mathbf{W}^t & \boldsymbol{\xi}_{n_yN}^{b^\top} \boldsymbol{\zeta}^t \\ 
    \end{bmatrix}.
    \end{aligned}
    \end{equation}
\end{lemma}
\begin{proof} To prove the statement we consider an arbitrary output $[\mathbf{y}_k]_i$, i.e., the $i$-th element of the output vector $\mathbf{y}_k$. Then \eqref{eqn:multi_step_output} becomes 
    \begin{align*}
        [\mathbf{y}_k]_i &= \left( \mathbf{W}_i^b   \Phi_b(\mathbf{u}_k) + 
 \boldsymbol{\xi}_i^b \right)^\top \left( \mathbf{W}^t   \Phi_t(x_k) + 
 \boldsymbol{\zeta}^t \right) \\ 
 &= \left(\Phi_b^{\top}(\mathbf{u}_k) \mathbf{W}_i^{b^\top} 
 + \boldsymbol{\xi}_i^{b^\top} \right) \left( \mathbf{W}^t   \Phi_t(x_k) + 
 \boldsymbol{\zeta}^t \right) \\
 & = \Phi_b^{\top}(\mathbf{u}_k) \mathbf{W}_i^{b^\top} \mathbf{W}^t   \Phi_t(\mathbf{x}_0) + \Phi_b^{\top}(\mathbf{u}_k) \mathbf{W}_i^{b^\top} \boldsymbol{\zeta}^t  \\
 &  \quad + \boldsymbol{\xi}_i^{b^\top} \mathbf{W}^t \Phi_t(x_k) + \boldsymbol{\xi}_i^{b^\top} \boldsymbol{\zeta}^t \\
  & = \Phi_b^{\top}(\mathbf{u}_k) \mathbf{W}_i^{b^\top} \mathbf{W}^t   \Phi_t(x_k) + \boldsymbol{\zeta}^{t^\top}  \mathbf{W}_i^{b}  \Phi_b(\mathbf{u}_k) \\
  & \quad + \boldsymbol{\xi}_i^{b^\top} \mathbf{W}^t \Phi_t(x_k) + \boldsymbol{\xi}_i^{b^\top} \boldsymbol{\zeta}^t.
    \end{align*}
    We can further rewrite the first term by using the the properties of the Kronecker product and the vectorization operator, i.e., 
    \begin{align*}
        &\Phi_b^{\top}(\mathbf{u}_k) \mathbf{W}_i^{b^\top} \mathbf{W}^t   \Phi_t(x_k)  \\
        &\quad \quad  \quad \quad =\text{vec}(\Phi_b^{\top}(\mathbf{u}_k) \mathbf{W}_i^{b^\top} \mathbf{W}^t   \Phi_t(x_k)) \\
        & \quad  \quad \quad  \quad = (\Phi_t^{\top}(x_k) \otimes  \Phi_b^{\top}(\mathbf{u}_k) ) \text{vec}(\mathbf{W}_i^{b^\top} \mathbf{W}^t)  \\
         &\quad \quad \quad  \quad = \text{vec}^{\top}(\mathbf{W}_i^{b^\top} \mathbf{W}^t) (\Phi_t(x_k) \otimes  \Phi_b(\mathbf{u}_k) ).
    \end{align*}
    If we repeat the same procedure for the other indices $i=1,\dots,n_yN$ we get the expression for $\Theta_o$ in \eqref{eqn:basis}, which finalizes the proof.
\end{proof}
$\Phi_{\otimes}(\mathbf{u}_k,x_k)$ is called a candidate basis because it does not necessarily satisfy the linear independence condition, even though this is highly likely to happen after training; see, e.g. \cite{Widrow_2013}; future work will deal with extracting a basis via proper orthogonal decomposition, in case the linear independence condition does not hold. 

\begin{remark}
        From the proof of Lemma~\ref{lemma:basis}, it is straightforward to show that the underlying basis for the standard unstacked DeepOnet in the SISO case is of the form
\begin{equation}
    y_{j|k} = \theta_o \phi_\otimes(\bu_k,z_{j,k}):= \theta_o \begin{bmatrix}
        \phi_b(\mathbf{u}_k)\otimes \phi_t(z_{j,k}) \\
        \phi_b(\mathbf{u}_k) \\
        \phi_t(z_{j,k}) \\
        1
        \end{bmatrix},
\end{equation}
where $\theta_o\in\mathbb{R}^{1\times n_tn_b + n_t +n_t+1}$. Note that this is not necessarily true for the MIMO case. 
\end{remark}
\begin{remark}
\label{rem:alt:base}
    By using the properties of the Kronecker product, the basis representation \eqref{eqn:output_basis} can be equivalently represented as
    \begin{equation}
        \mathbf{y}_k = \mathbf{\Theta}_o(x_k) \begin{bmatrix}
          \Phi_b(\mathbf{u}_k) \\
            1
        \end{bmatrix},
    \end{equation}   
    where  
    \begin{equation}
    \mathbf{\Theta}_o(x_k) := \Theta_o  \begin{bmatrix}
            \left( I_{n_t} \otimes \Phi_t(x_k)\right) & 0 \\
            I_{n_b} & 0 \\
            0 & \Phi_t(x_k) \\
            0 & 1 
        \end{bmatrix}. 
    \end{equation}
After substituting the initial condition $x_k$ into $\mathbf{\Theta}_o(x_k)$ becomes a constant matrix in $\mathbb{R}^{Nn_y \times n_b+1}$, the dimension of the prediction model is completely independent of the dimensions of the trunk net. 
\end{remark}
By deriving an explicit candidate basis for the MS-DeepONet architecture as in \eqref{eqn:output_basis} we revealed the underlying structure and computational advantages. Moreover, this opens up the possibility of using the MS-DeepONet for data-enabled predictive control (DeePC), see \cite{Basis_Lazar}. In the next section, we present a framework for training the MS-DeepONet from data.

\section{MULTI-STEP DEEPONET LEARNING AND HYPER PARAMETER TUNING ALGORITHMS}\label{sec:ablation_study}

This section describes the model learning approach and the process for selecting optimal hyper parameters for the MS-DeepONet. A flowchart illustrating the relationships between the relevant algorithms is provided in Figure~\ref{fig:flow_chart_algorithms}.
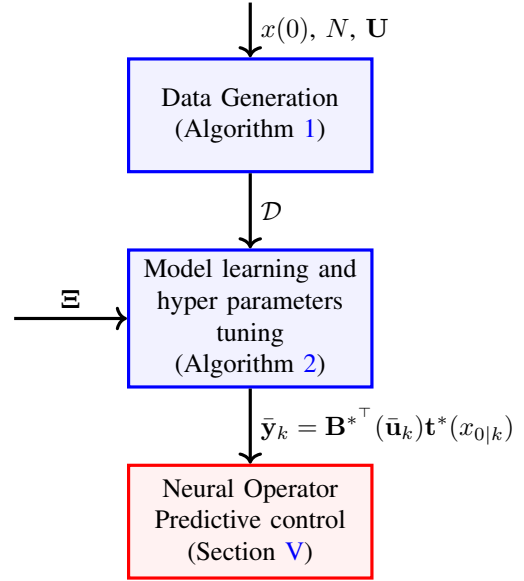
\begin{figure}[h!]
        \centering
        \begin{center}
\begin{tikzpicture}[
Block1/.style={draw=blue, fill = blue!5, very thick,  minimum width=2cm, minimum height=1.5cm},
Block2/.style={draw=red, fill = red!5, very thick,  minimum width=2cm, minimum height=1.5cm},
Block3/.style={draw=red, fill = red!5, very thick,  minimum width=1.2cm, minimum height=1.2cm},
Block4/.style={draw=black, fill = black!5, very thick,  minimum width=1.5cm, minimum height=1.5cm},
node distance=10mm  
]
 \node[Block1] (block1) {\parbox{3cm}{\centering Data Generation \\ (Algorithm \ref{alg:1})}};
 \node[Block1] (block2) [below=of block1] {\parbox{3cm}{\centering Model learning and hyper parameters tuning \\ (Algorithm \ref{alg:2})}};


    
    

\node[Block2] (block3) [below=of block2] {\parbox{3cm}{\centering Neural Operator Predictive control \\ (Section \ref{sec6})}};

 \draw[->, very thick] (block2.west) ++(-15mm,0) -- (block2.west) node[midway, above] {$\mathbf{\Xi}$};
 \draw[->, very thick] ++(0,15mm) -- (block1.north) node[midway, right] {$x(0)$, $N$, $\mathbf{U}$};
 \draw[->, very thick] (block2.south) -- (block3.north) node[midway, right] {$\bar\by_k =\mathbf{B}^{*^\top}(\bar \bu_k)\mathbf{t}^*(x_{0|k})$};
 \draw[->, very thick] (block1.south) -- (block2.north) node[midway, right] {$\mathcal{D}$};

\end{tikzpicture}
\end{center}
        \caption{Flow chart for model learning algorithms and predictive control. Red colors indicate online algorithms and blue colors offline.}
        \label{fig:flow_chart_algorithms}
\end{figure}
\subsection{Data-generation}\label{subsec:data_gen}
Let $\mathbf{U} := \col(u(0), u(1), \dots, u(T+N-2))$ denote the input signal over the experiment duration, where $T \in \mathbb{N}$ is the batch size during training, and $N \in \mathbb{N}$ is the prediction horizon. Given an initial condition $x(0)$, state and output data $\mathbf{x}$ and $\mathbf{y}$ is generated by applying $\mathbf{u}$ to the system~\eqref{eqn:systems} as outlined in Algorithm~\ref{alg:1}. The structure of the data matrices depends on the type of prediction model that is learned. Specifically, for the MS-DeepONet model, described in subsection~III.\ref{subsec:MS_architecture}, the data matrices are defined as follows: $\mathcal{U} = \mathcal{H}^0_{[N,T]}(\mathbf{U})$, $\mathcal{Y} = \mathcal{H}^1_{[N,T]}(\mathbf{Y})$, $\mathcal{Z} = \mathcal{H}^0_{[1,T]}(\mathbf{X})$. The input generation can be done according to nonlinear system identification guidelines. The input and output dimensions of the sub-networks in the proposed DeepONet architecture during training are shown in TABLE~\ref{tab:dimensions2}.
\begin{algorithm}[t!]
\caption{Data-generation for MS-DeepONet}\label{alg:1}
\begin{algorithmic}[1]
\State \textbf{Input:} 
\State \quad $x(0)$: (\textcolor{black}{initial condition})
\State \quad $N$: (\textcolor{black}{prediction horizon})
\State \quad $\mathbf{U}$: (\textcolor{black}{input data samples})
\State \textbf{Output:} 
\State \quad $\mathcal{D}$: (\textcolor{black}{data set for model learning})
\State \textbf{Step 1:} Generate piecewise constant input signal:
\State \quad $\{u(j)\}_{j\in[iT_s,(i+1)T_s)} = [\mathbf{U}]_i, \quad \forall i\in [0,T+N-2]$
\State \textbf{Step 2:} For time interval for $t \in [0,(T+N-2)T_s)$, perform experiment for system~\eqref{eqn:systems}:
\State \quad $x(0) = \mathbf{x}_0$ (set initial condition)
\State \quad Apply input $u(t)$ to system  
\State \quad $[\mathbf{X}]_i := x(i T_s)$ (append state data sample)
\State \quad $[\mathbf{Y}]_i = y(i T_s)$ (append output data sample)
\State \textbf{Step 3:} Construct the data set:
    \State \quad $\mathcal{D} = \{\mathcal{Y}, \mathcal{U}, \mathcal{Z}\}$
\end{algorithmic}
\end{algorithm}
In the next subsection we present the learning algorithm and hyper parameter tuning approach. 
\subsection{Model learning and hyper parameter tuning}\label{subsec:ablation}
This section discusses the prediction model learning and hyper parameter tuning approach for the MS-DeepONet. We firstly fix a set of tunable hyper parameters. To this end, consider the expressions for the hidden layers of the branch and trunk networks respectively:
\begin{subequations}
\label{eqn:hidden_layers}
\begin{align}
 h_i^b &= \sigma(W_{i-1}^b h_{i-1}^b + b_{i-1}^b), \quad i = 1,\dots,l_b,\label{eqn:hidden_layers_a} \\
h_i^t &= \sigma(W_{i-1}^t h_{i-1}^t + b_{i-1}^t), \quad i = 1,\dots,l_t, \label{eqn:hidden_layers_b}
\end{align}
\end{subequations}
where the weight matrices and bias vectors for the trunk net are defined as $W_i^t \in \mathbb{R}^{n_{t,i}\times n_{t,i+1}}$ and $b^t_i \in \mathbb{R}^{n_{t,i}}$, with $i=0,\dots,l_t-1$ and $n_{t,0}=n_x$ and $n_{t,l_t}=n_t$. Similarly, the weight matrices and bias vectors for the branch net are defined as $W_i^b \in \mathbb{R}^{n_{b,i}\times n_{b,i+1}}$ and $b^b_i \in \mathbb{R}^{n_{b,i}}$, with $i=0,\dots,l_b-1$ and $n_{t,0}=Nn_u$ and $n_{b,l_b}=n_b$. The function $\sigma$ represents an arbitrary activation function that is applied in an element-wise fashion to any tensor of arbitrary dimension. 

An overview of the variable and dimensions of all the layers in the network is provided in Table~\ref{tab:dimensions1} and  Table~\ref{tab:dimensions2} respectively. In order to find the optimal set of hyper parameters we will perform an ablation study where we perform a grid search over a set of hyper parameter configurations. In this regard, for any hyper parameter $a$ in Table~\ref{tab:dimensions1}, let $a^i$ denote the value of $a$ in the $i$-th hyper parameter configuration. Then, we denote a single hyper parameter configuration by $\Xi_i = \left\{ \left(l_{\text{b}}^i,l_{\text{t}}^i , p^i, n^i_{b,0}, \dots,n^i_{b,l^i_b}, n^i_{t,0}, \dots,n^i_{t,l^i_t} \right) \right\}$. Let $\mathbf{\Xi} := \{\Xi_i\}$, where $i = 1, \dots, N_c$ and $N_c \in \mathbb{N}$ represent the number of hyper parameter configurations. Next, let $\mathbf{B}(\mathbf{u}_k,\Xi_i,\Theta_i^t)$ denote the branch output equation in matrix form \eqref{eqn:branch_matrix} for hyper parameter configuration $\Xi_i$ and let $\Theta^b_i$ be the set of all weights and biases of the hidden layers and output layer of the branch network. Note in this case that $\Phi_b(\mathbf{u}_k) = h_{l_b^i}^b$ and $h^b_0=\mathbf{u}_k$. Similarly, let $\mathbf{t}(x_k,\Xi_i,\Theta^t_i)$ denote the trunk output equation in vector form \eqref{eqn:trunk_matrix} for hyper parameter configuration $\Xi_i$ and let $\Theta^t_i$ be the set of all weights and biases of the hidden layers and output layer of the trunk network. Note in this case that $\Phi_b(\mathbf{u}_k) = h_{l_b^i}^b$ and $h^b_0=\mathbf{u}_k$. Finally, the hyper parameter dependent prediction model is defined as:
\begin{equation}\label{eqn:output_MS_DeePONet_as}
   \bar\by_k := \mathbf{B}^\top(\bar\bu_k,\Xi_i,\Theta^b_i) \mathbf{t}(x_{0|k},\Xi_i,\Theta^t_i),
\end{equation}
\begin{table}[t!]
\caption{MS-DeepONet model architecture parameters.}
\label{tab:dimensions1}
\begin{center}
\begin{tabular}{ |c|c| }
\hline
\textbf{Symbol} & \textbf{Description} \\
\hline
$l_b$ & Number of hidden layers in branch net \\ 
$l_t$ & Number of hidden layers in trunk net \\ 
$n_{b,i}$ & Neurons in hidden layer $i$ of branch net \\ 
$n_{t,i}$ & Neurons in hidden layer $i$ of trunk net \\ 
$p$ &  Output layers dimension \\
$N$ &  Prediction horizon \\
\hline
\end{tabular}
\end{center}
\end{table}
\begin{table}[t!]
\caption{Model Architecture: Inputs and output dimensions for branch, trunk, and Product layers during training.}
 \label{tab:dimensions2}
\begin{center}
\begin{tabular}{ |c|c|c| }
\hline
& \textbf{Input dim.} & \textbf{Output dim.} \\
\hline
Branch Hidden layer $1$ & ($T$,$n_u N$) & ($T$,$n_{b,1}$) \\ 
Branch Hidden layer $i$ & ($T$,$n_{b,i-1}$) & ($T$,$n_{b,i}$) \\ 
Branch Output layers & ($T$,$n_{b,l_b}$) & ($T$,$n_y N$,$p$)\\ 
\hline
Trunk Hidden layer $1$ & ($T$,$n_x$) & ($T$,$n_{t,1}$) \\ 
Trunk Hidden layer $i$ & ($T$,$n_{t,i-1}$) & ($T$,$n_{t,i}$)\\ 
Trunk Output layer & ($T$,$n_{t,l_t}$) & ($T$,$p$)\\ 
\hline
Product layer & ($T$,$p$) $\times$ ($T$,$n_yN$,$p$) & ($T$,$n_yN$) \\ 
\hline
\end{tabular}
\end{center}
\end{table}
We perform the ablation study defined in Algorithm~\ref{alg:2} to select the optimal prediction model over all configurations in $\mathbf{\Xi}$. Firstly, the data set $\mathcal{D}$ is divided into the training subset ($\mathcal{D}^{\text{train}}$) and validation ($\mathcal{D}^{\text{val}}$) subsets. Next, we define the model learning problem for hyper parameter configuration $\Xi_i\in\mathbf{\Xi}$.
\begin{algorithm}[t!]
\caption{Ablation Study}\label{alg:2}
\begin{algorithmic}[1]
\State \textbf{Input:} 
\State \quad $\mathcal{D}$: (data set for model learning)
\State \textbf{Output:} 
\State \quad $\hat{\mathcal{G}}^{*}$: (Optimal prediction model)
\State \textbf{Step 1:} split data $\mathcal{D} = \{\mathcal{Y}, \mathcal{U}, \mathcal{Z}\}$ into training $\mathcal{D}^{\text{train}} = \{\mathcal{Y}^{\text{train}}, \mathcal{U}^{\text{train}}, \mathcal{Z}^{\text{train}}\}$ and validation $\mathcal{D}^{\text{val}} = \{\mathcal{Y}^{\text{val}}, \mathcal{U}^{\text{val}}, \mathcal{Z}^{\text{val}}\}$ data sets. 
\State \textbf{Step 2:} Train models for hyper parameter configurations:
\For{$\Xi_i \in \mathbf{\Xi}$}
    \State \quad \textbf{Step 2.1:} Solve Problem~\ref{prob:NNOpt} for $\Xi_i$ 
    \State \quad \textbf{Step 2.2:} Construct the estimate validation data:
    \State \quad \quad \quad $\hat{\mathcal{Y}}^{\text{val}}_i = \mathbf{B}^\top(\mathcal{U}^{\text{val}},\Xi_i,\Theta_i^{b^*})\mathbf{t}(\mathcal{Z}^{\text{val}},\Xi_i,\Theta_i^{t^*})$, 
    \State \quad \textbf{Step 2.2:} Calculate the validation loss:
    \State \quad \quad \quad $J^{\text{val}}_i := J^{\text{val}}(\mathcal{Y}^{\text{val}}, \hat{\mathcal{Y}}^{\text{val}}_i)$
\EndFor
\State \textbf{Step 3:} Identify index $I = \arg \min_{i} J_i^{\text{val}}$ minimizing the validation loss and choose $\bar \by_k = \mathbf{B}^{*^\top}(\bar\bu_k)\mathbf{t}^*(x_{0|k})$, where 
 \begin{align*}
     \mathbf{B}^{*}(\bar\bu_k) &:= \mathbf{B}(\bar\bu_k,\Xi_I,\Theta_I^{b^*}), \\
     \mathbf{t}^*(x_{0|k}) &:= \mathbf{t}(x_{0|k},\Xi_I,\Theta_I^{t^*}).
 \end{align*} 
\end{algorithmic}
\end{algorithm}
\begin{problem} (MS-DeepONet prediction model learning problem for $\Xi_i$.) \label{prob:NNOpt}
\begin{subequations}
\label{eq:3:1_2}
\begin{align}
& \{\Theta_i^{b^*}, \Theta_i^{t^*} \} = \arg \min_{\Theta_i^{b}, \Theta_i^{t}} \quad J(\mathcal{Y}^{\text{train}}, \hat{\mathcal{Y}}^{\text{train}}_i,\Theta) \\
& \text{subject to:} \nonumber \\
& \quad \hat{\mathcal{Y}}^{\text{train}}_i = \mathbf{B}^\top(\mathcal{U}^{\text{train}},\Xi_i,\Theta_i^b)\mathbf{t}(\mathcal{Z}^{\text{train}},\Xi_i,\Theta_i^t),
\end{align}
\end{subequations}
\end{problem}
where $\Theta_i = \Theta_i^t\cup \Theta^b_i$ collects all the weights and biases of the MS-DeepONet model for hyper parameter configuration $\Xi_i$. The problem above will be solved $N_c$ times for all hyper parameter configurations within $\mathbf{\Xi}$. As a cost function we typically use the the regularized least squares cost function:
\begin{equation}\label{eqn:loss}
    J(\mathcal{Y}, \hat{\mathcal{Y}}, \Theta) = \frac{\| \mathcal{Y} - \hat{\mathcal{Y}}\|_2^2}{\|\mathcal{Y} \|_2^2}+ \lambda\|\Theta\|_2^2.
\end{equation}
\begin{remark}
    The regularization term $\lambda \| \Theta \|_2^2$ in \eqref{eqn:loss} is used to prevent overfitting. The parameter $\lambda>0$ is typically referred to as the weight decay.  
\end{remark}
In order to select the prediction model over all hyper parameter configurations, we firstly calculate the validation output data, i.e.,
\begin{equation}\label{eqn:validation}
\hat{\mathcal{Y}}^{\text{val}}_i = \mathbf{B}^{*^\top}(\mathcal{U}^{\text{val}},\Xi_i,\Theta_i^b)\mathbf{t}^*(\mathcal{Z}^{\text{val}},\Xi_i,\Theta_i^t),
\end{equation}
and next we check the validation loss, i.e.,
\begin{equation}
    J^{\text{val}}(\mathcal{Y}^{\text{val}}, \hat{\mathcal{Y}}^{\text{val}}_i) = \frac{\| \mathcal{Y}^{\text{val}} - \hat{\mathcal{Y}}^{\text{val}}_i\|_2^2}{\|\mathcal{Y}^{\text{val}} \|_2^2}.
\end{equation}
 The algorithm~\ref{alg:2} summarizes the ablation study to select the optimal prediction model prediction model $\bar\by_k = \mathbf{B}^{*^\top}(\bar\bu_k)\mathbf{t}^*(x_{0|k})$. Next, we show how to incorporate the model into a predictive control scheme to actively control the dynamical system \eqref{eqn:systems}.

\section{NEURAL OPERATOR PREDICTIVE CONTROL}
\label{sec6}
Model predictive control is an optimization-based strategy that computes a control input sequence $\bar \bu_k = \col(u_{0|k}, \dots, u_{N-1|k})$ by solving a finite-horizon constrained optimal control problem at each time step $t = kT_s$, where $k = 0, 1, 2, \dots$, using the current measured output (or state) $y(kT_s)$ of a dynamical system. The first input $u(kT_s) = u_{0|k}$ in the sequence is applied to the system, typically by means of Zero-Order-Hold, as shown in \figurename~\ref{fig:neural_pc_online}. This section introduces several implementations based on the both the standard DeepONet and the MS-DeepONet architecture. Additionally, we propose a data-enabled predictive controller for the MS-DeepONet, utilizing the candidate basis derived in subsection~III.\ref{subsec:basis}. 
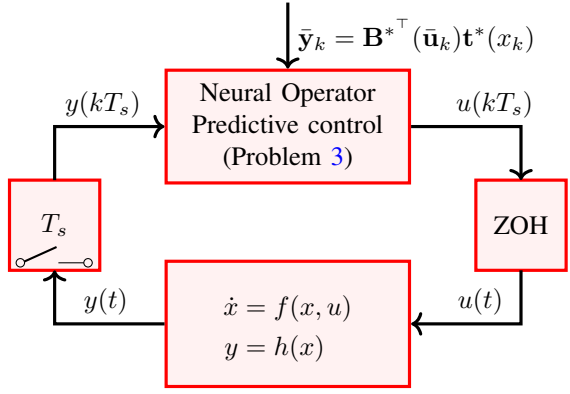
\begin{figure}[t!]
        \centering
        \begin{center}
\begin{tikzpicture}[
Block1/.style={draw=blue, fill = blue!5, very thick,  minimum width=2cm, minimum height=1.5cm},
Block2/.style={draw=red, fill = red!5, very thick,  minimum width=2cm, minimum height=1.5cm},
Block3/.style={draw=red, fill = red!5, very thick,  minimum width=1.2cm, minimum height=1.2cm},
Block4/.style={draw=black, fill = black!5, very thick,  minimum width=1.5cm, minimum height=1.5cm},
node distance=10mm  
]
       
\node[Block2] (block3) {\parbox{3cm}{\centering Neural Operator Predictive control \\ (Problem~\ref{prob:MS_DeepONet})}};
\node[Block2] (block4) [below=of block3] {\parbox{3cm}{\centering \begin{align*}
    \dot x &= f(x,u) \\
    y &= h(x)
\end{align*}}};

\node[Block3] (block1_small) [right=of $(block3)!0.5!(block4)$, xshift=15mm, inner sep=0pt, outer sep=0pt] {ZOH};
 
\node[Block3, inner sep=0pt, outer sep=0pt, minimum width=1.2cm, minimum height=1.2cm] (block2_small) [left=of $(block3)!0.5!(block4)$, xshift=-15mm] { 
   \begin{circuitikz}
     \draw (0.1,0.2) to[nos, o-o] (1.0,0.2);  
     \node[above=0.0cm] at (0.5, 0.1) {$T_s$};  
   \end{circuitikz}
};

 \draw[->, very thick] (block4.west) -| (block2_small.south) node[midway, right, xshift=2.5mm, yshift=3mm] {$y(t)$};

 \draw[->, very thick] (block2_small.north) |- (block3.west) node[midway, right, xshift=0mm, yshift=3mm] {$y(kT_s)$};

\draw[->, very thick] (block3.east) -| (block1_small.north) node[midway, right, xshift=-10.0mm, yshift = 3mm] {$u(kT_s)$};

\draw[->, very thick] (block1_small.south) |- (block4.east) node[midway, left, xshift= -1.0mm, yshift = 3.0mm] {$u(t)$};

\draw[->, very thick] 
  ([yshift=25pt]block3.north) -- (block3.north) 
  node[midway, right] 
  {$\bar \by_k = \mathbf{B}^{*^\top}(\bar\bu_k)\mathbf{t}^*(x_k)$};

\end{tikzpicture}
\end{center}
        \caption{Neural Operator Predictive Control.}      
        \label{fig:neural_pc_online}
\end{figure}

\subsection{Multi-step DeepONet MPC problem formulation}\label{lst:yalmip_ssDeepONet}
Next we present the predictive control problem for the MS-DeepONet formulation which will be solved online for all $t=kT_s$, wher $k\in\mathbb{N}$.
\begin{problem} (MS-DeepONet MPC Problem at time $t=kT_s$) \label{prob:MS_DeepONet}
\begin{subequations}
\label{eq:3:1_3}
\begin{align}
&\min_{\bar\bu_k}  \quad   (\bar\by_k-\mathbf{r})^{\top}\Omega(\bar\by_k-\mathbf{r}) + \Delta \bar\bu^\top_{k}\Psi \Delta \bar\bu_{k}  \label{eq:DMPCp_a2} \\ 
&\text{subject to: } \nonumber  \\
& \quad x_{0|k} =  x(kT_s), \\
& \quad \bar \by_k = \mathbf{B}^{*^\top}(\bar\bu_k)\mathbf{t}^*(x_{0|k}), \label{eq:DMPC_MS_b} \\
&\quad (\bar\by_k , \bar\bu_k)  \in \Yset^{Nn_y} \times \Uset^{Nn_u}. \label{eq:DMPC_MS_c}
\end{align}
\end{subequations}
\end{problem}
We define a quadratic cost function as follows:
\begin{equation}\label{eqn:quadratic_cost}
    \Omega = \begin{bmatrix}
        Q & & & \\
         & \ddots &  & \\
          &  & Q & \\
        &  &   & P\\
    \end{bmatrix}, \quad \Psi = \begin{bmatrix}
        R & &  \\
         & \ddots &   \\
          &  & R  
    \end{bmatrix},
\end{equation}
where $Q\succ0$, $P\succ0$ and $R\succ0$ are strictly positive definite matrices. In the problem above, we penalize the vector \(\Delta \mathbf{u}_k := \col(\Delta u_{0|k}, \dots, \Delta u_{N-1|k})\), where \(\Delta u_{i|k} := u_{i|k} - u_{i-1|k}\) with \(u_{-1|k} = u((k-1)T_s)\) the applied input from the previous sampling instance. This approach enables us to stabilize the system at various equilibria without the need for computing a reference signal for the input. A popular choice for the constraint sets $\mathbb{Y}$ and $\mathbb{U}$ are:
\begin{align}
    \Yset &= \{y_{i|k} \mid \mathbf{A}_y y_{i|k} \leq \mathbf{b}_y\}, \quad \forall i=1,\dots,N,\\
    \Uset &= \{u_{i|k} \mid \mathbf{A}_u u_{i|k} \leq \mathbf{b}_u\}, \quad \forall i=0,\dots,N-1.
\end{align}
It is worth pointing out that the MS-DeepONet basis representation derived in Remark~\ref{rem:alt:base} can be used in \eqref{eq:DMPC_MS_b}, that is, \[\bar\by_k=\Theta_o\Phi_{\otimes}(\bar\bu_k,x_{0|k})=\mathbf{\Theta}_o(x_{0|k}) \begin{bmatrix}
          \Phi_b(\bar\bu_k) \\
            1
        \end{bmatrix}.\]
\begin{remark}
    For the problem formulation, we assume that the state $x(kT_s)$ is known, at each sampling instant. A practical alternative is to use shifted inputs and outputs as the state, i.e., $\col(u(k-1)T_s,\dots,\col(u(k-T_{\text{ini}})T_s), y(k)T_s,\dots,\col(u(1+k-T_{\text{ini}})T_s))$ for some $T_{\text{ini}}\in\mathbb{N}$. In this case, the data matrix $\mathcal{Z}$ must be adapted accordingly during the model learning phase. 
\end{remark}

\begin{remark}
With the extracted candidate basis \eqref{eqn:output_basis} it is possible to use the MS-DeepONet in a data-enabled predictive control scheme \cite{Coulson2019, Neural_Basis_Lazar}. To do this, replace \eqref{eq:DMPC_MS_b} in Problem~\ref{prob:MS_DeepONet} by:
\begin{equation*}
    \begin{bmatrix}
    \boldsymbol{\Phi}_{\otimes} \\
    \mathbf{\mathcal{Y}}
\end{bmatrix} \bg_k = 
\begin{bmatrix}
    \Phi_{\otimes}(\bar\bu_k,x_{0|k}) \\
    \bar\by_k
\end{bmatrix}.
\end{equation*}
Here, $\bg_k \in \mathbb{R}^T$ is a vector of optimization variables that allows for some additional freedom when predicting the future output via the MS-DeepONet, compared to the uniquely defined MS-DeepONet predictor in \eqref{eq:DMPC_MS_b}, at the cost of an increase in computational complexity due to the additional variables. Typically, a regularization term  $l_g(g_k) := \lambda \| \bg_k \|_2^2$ is added to the cost, where  $\lambda > 0$. Moreover, $\boldsymbol{\Phi}_{\otimes} := \Phi^*_{\otimes}(\mathcal{U},\mathcal{Z})\in \mathbb{R}^{n_tn_b + n_t + n_b +1\times T}$ corresponds to the data matrix generated by feeding the columns of $\mathcal{U}$ and $\mathcal{Z}$ through the learned MS-DeepONet basis. The online implementation of the corresponding problem is more challenging due to the data-dependent size of $\boldsymbol{\Phi}_{\otimes}$, $\mathbf{\mathcal{Y}}$ and $\mathbf{g}_k$. Further analysis and implementation of MS-DeepONet data-enabled predictive control is beyond the scope of this paper and will be considered in future work.
\end{remark}
For completeness, we also provide the MPC problem formulation and model learning procedure for the MimoONet in the next subsection.  

\subsection{Standard DeepONet model learning approach and MPC problem formulation}\label{seubsec_5B_DeepONet}
The data generation process for the standard unstacked DeepONet is similar to the MS-DeepONet data generation. In Algorithm~\ref{alg:1}, simply replace the data matrices by:
\begin{align*}
    &\mathcal{Y} = \begin{bmatrix}
        y(T_s) \ldots y(NT_s) & y(2T_s) \dots  y((N+1)T_s) & \dots 
    \end{bmatrix}, \\
    &\mathcal{Z} = \begin{bmatrix}
        x(0)  \dots x(0) & x(T_s)  \dots  x(T_s) & \dots  \\
        T_s  \dots  NT_s & T_s  \dots  NT_s  & \dots 
    \end{bmatrix}, \\
    &\mathcal{U}_j = \begin{bmatrix}
        \mathbf{\bar u}^j_0 \dots \mathbf{\bar u}^j_0 & \mathbf{\bar u}^j_{T_s}  \dots  \mathbf{\bar u}^j_{T_s}  & \dots 
    \end{bmatrix},
\end{align*}
where $\mathbf{\bar u}^j_i = \col(u_j(iT_s),\dots,u_j((i+N-1)T_s))$ for all $j=1,\dots,n_u$. Recall that the standard DeepONet is given by \eqref{eqn:MimoOnet_pm}. An ablation study can be performed similarly as for the MS-DeepONet as explained in Algorithm~\ref{alg:2}. Next we present the predictive control problem for the standard DeepONet formulation which will be solved online for all $t=kT_s$, where $k\in\mathbb{N}$.
\begin{problem} (Standard DeepONet MPC Problem at time $t=kT_s$) \label{prob:DMPC_DeePONet_standard}
\begin{subequations}
\label{eq:3:1_4}
\begin{align}
&\min_{\bar\bu_k}  \quad   (\bar\by_k-\mathbf{r})^{\top}\Omega(\bar\by_k-\mathbf{r}) + \Delta \bar\bu^\top_{k}\Psi \Delta \bar\bu_{k}  \label{eq:DMPCp_a3} \\ 
&\text{subject to: } \nonumber  \\
& \quad x_{0|k} =  x(kT_s), \\
& \quad y_{j|k} = \sum_{l=1}^{p} \mathbf{b}^{n_u}_{l,i}(\bar\bu_k)
   t_l(x_{0|k},jT_s), \quad j = 1,\dots, N, \label{eq:DMPCp_b2} \\
& \quad (\bar \by_k ,\bar \bu_k)  \in \Yset^{Nn_y} \times \Uset^{Nn_u}. \label{eq:DMPCp_c1}
\end{align}
\end{subequations}
\end{problem}
In the next section we analyze the performance Problem~\ref{prob:MS_DeepONet} and Problem~\ref{prob:DMPC_DeePONet_standard} by means of tree nonlinear benchmark examples. 

\section{NUMERICAL EXAMPLES}
\label{sec:examples}
We evaluate the performance of the proposed MPC schemes using the MS-DeepONet and standard DeepONet (Problems~\ref{prob:MS_DeepONet} and~\ref{prob:DMPC_DeePONet_standard}). To ensure a fair comparison, both use identical datasets, MPC costs, prediction horizons, and comparable network dimensions. The code for all simulations is available on \href{https://github.com/todejong/Deep-Operator-Neural-Network-Model-Predictive-Control.git}{GitHub}. We first assess tracking performance on the nonlinear SISO van der Pol oscillator, then evaluate the nonlinear MIMO quadruple tank process, and finally address the cart-pendulum swing-up task, which requires closed-loop data due to its unstable upright equilibrium. An ablation study on the van der Pol example investigates the influence of key hyper parameters. Table~\ref{tab:performance_all} summarizes network configurations, MPC performance (average computation time and absolute mean error), and training statistics (training time, training loss, and validation loss). All continuous-time systems are simulated using the \texttt{solve\_ivp} function from \texttt{SciPy} with an explicit Runge--Kutta method of order $5(4)$. MPC problems are solved using \texttt{CasADi} with the \texttt{IPOPT} solver. All results presented in this section were generated using a 13\textsuperscript{th} Gen Intel\textsuperscript{\textregistered} Core\textsuperscript{TM} i7-1370P CPU running at 1.90~GHz. No GPU acceleration has been used to enhance performance. 
\begin{figure}[b!]
  \centering
\includegraphics[width=\columnwidth]{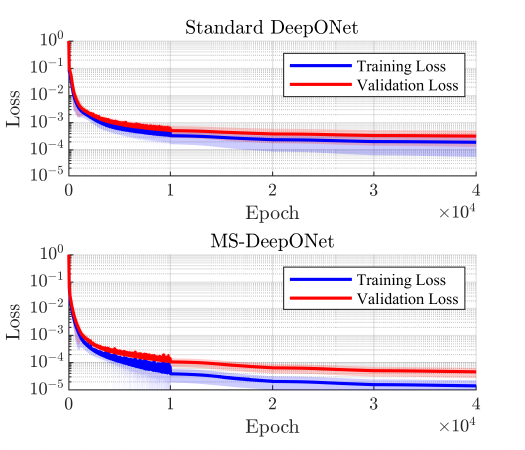}\caption{Training and validation loss and standard deviation for MS-DeepONet learning over the ablation study for van der Pol oscillator.}
  \label{fig:train_val_vdp}
\end{figure}

\paragraph{Van der Pol oscillator}
We asses the effectiveness of the MS-DeepONet and standard DeepONet (Problems~\ref{prob:MS_DeepONet} and~\ref{prob:DMPC_DeePONet_standard}) by using the van der Pol oscillator with control input, i.e.,
\begin{subequations}
\label{eqn:vanderpol}
\begin{align}
    \dot{x}_1 &= x_2, \\
    \dot{x}_2 &= \mu(1-x_1^2)x_2 + u, \\
    y &= x_1,
\end{align}
\end{subequations}
with $\mu=1$. An open-loop identification experiment is performed using a piecewise constant input signal with a sampling time of $T_s=0.1$s with $2000$ samples. The datasets are generated as explained in Algorithm~\ref{alg:1} and a total of $40000$ epochs is used for training the models. 

We perform an ablation study in order to guarantee optimality of the hyper parameters, see Algorithm~\ref{alg:2}. This amounts to performing a grid search on a set of hyper parameter configurations: $\mathbf{\Xi} = \left\{ (l_{\text{b}},l_{\text{t}} , p, n_{b,1},\dots,n_{b,l_b}, n_{t,1},\dots,n_{t,l_t},) \right\} $ where $l_{\text{b}} = l_{\text{t}} \in \{1, 2, 3\}$, $p \in \{20, 30, 40\}$ and $n_{b,i} = n_{t,i} \in \{20, 30, 40\}$ for $i = 1, \dots, l$, where $l = l_b = l_t$. The total number of combinations is $|\mathbf{\Xi}| = 27$. The mean training loss and the mean validation loss for all hyper parameter configurations during training are shown in \figurename~\ref{fig:train_val_vdp}. The shaded area around the mean is the standard deviation over the $27$ combinations. Clearly, the MS-DeepONet architecture achieves a lower training and validation loss over all hyper parameter configurations. The optimal set of hyper parameters based on the ablation study, is shown in Table~\ref{tab:performance_all}. These hyper parameters are used to generate the simulation results that are discussed next.
\begin{table*}[t!]
\centering
\caption{
Comparison of MS-DeepONet and Standard DeepONet across three systems: van der Pol, 4-Tank, and Pendulum-on-a-Cart. The table includes (i) model architecture details, (ii) MPC controller performance, and (iii) training metrics.
}
\begin{tabular}{|c||c|c|c|c|c|}
 \hline
 \multirow{2}{*}{\textbf{ }} &
      \multicolumn{2}{c|}{\textbf{van der Pol}} &
      \multicolumn{2}{c|}{\textbf{Quadruple Tank Process}} &
      \textbf{Pendulum on a cart} \\
      & DeepONet & MS-DeepONet & DeepONet & MS-DeepONet & MS-DeepONet \\
 \hline\hline
Neurons in output layers $p$ & $20$ & $20$ & $20$ & $20$ & $40$ \\ 
 \hline
Neurons per layers $n_{t,i}$ and $n_{b,i}$ & $\{40,40,40\}$ & $\{40,40,40\}$ & $\{20,20\}$ & $\{20,20\}$ & $\{128,256,128\}$ \\ 
 \hline
Hidden layers $l_{t}$ and $l_b$ & $3$ & $3$ & $2$ & $2$ & $3$ \\ 
 \hline
Prediction horizon $N$ & $10$ & $10$ & $20$ & $20$ & $40$ \\ 
\noalign{\hrule height 1.5pt} 
AME & $0.1856$ & \textcolor{green_dark}{$0.0958$} & $0.0314$ & \textcolor{green_dark}{$0.0156$} & $0.1529$ \\ 
 \hline
Computation time & $0.0087$s & \textcolor{green_dark}{$0.0078$s} & $0.1678$s & \textcolor{green_dark}{$0.1631$s} & $3.8718$s  \\ 
\noalign{\hrule height 1.5pt} 
Training time & $55$min & \textcolor{green_dark}{$14$min} & $164$min & \textcolor{green_dark}{$40$min} & $195$min \\ 
 \hline
Training loss & $3.9094\cdot10^{-5}$ & \textcolor{green_dark}{$5.3679\cdot10^{-6}$} & $2.9620\cdot10^{-5}$ & \textcolor{green_dark}{$9.5186\cdot10^{-6}$} & $0.0108$ \\ 
 \hline
Validation loss & $9.4211\cdot10^{-5}$ & \textcolor{green_dark}{$2.7489\cdot10^{-5}$} & $9.0668\cdot10^{-5}$ & \textcolor{green_dark}{$2.1338\cdot10^{-5}$} & $0.0127$ \\ 
 \hline
\end{tabular}
\label{tab:performance_all}
\end{table*}

\begin{figure}[b!]
  \centering
\includegraphics[width=0.95\columnwidth]{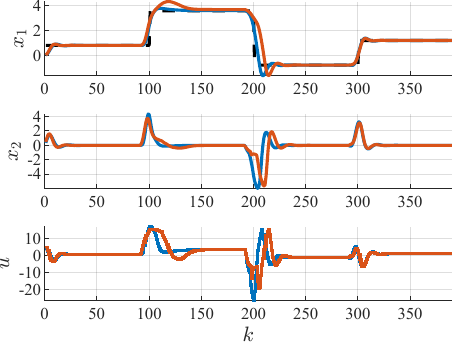}
\caption{Reference tracking for MS-DeepONet MPC (\textcolor{Matlab_blue}{--}) vs standard DeepONet MPC (\textcolor{Matlab_orange}{--}) for van der Pol oscillator.}
  \label{fig:tracking2}
\end{figure}

Both MPC schemes (Problems~\ref{prob:MS_DeepONet} and~\ref{prob:DMPC_DeePONet_standard}) are implemented with a prediction horizon $N=10$. $Q = 100$, $R=1$. Consider the closed--loop response for tracking a piecewise constant reference signal in \figurename~\ref{fig:tracking2}. In general, both MPC controllers are able to track the reference signal well for magnitudes smaller that $2$ for $x_1$. This is expected since $x_1$ does not become much larger than $2$ during training. Remarkably, for a reference magnitude of $4$, it seems that both controllers are still able to stabilize the system, indicating that the learned models generalize well outside of the data set. It seems that the MS-DeepONet based controller has an overall better tracking performance than the standard DeepONet based controller. This is confirmed by the absolute mean tracking error (AME) shown in Table~\ref{tab:performance_all}. Note also that both the training and validation loss of the MS-DeepONet architecture are lower. Notably, the average computation time per control action for both controllers is much shorter ($0.0087$s vs. $0.0078$s) than the sampling period $T_s = 0.1$s, highlighting the architecture's potential for real-time MPC implementation. Lastly, note that it takes significantly longer to train the standard DeepONet model than the MS-DeepONet model ($55$min vs. $14$min) for the exact same data size and number of epochs. Overall, MS-DeepONet consistently outperforms the standard DeepONet across all metrics: it achieves better validation and control performance, slightly lower average computation time, and significantly faster training.

\paragraph{Quadruple Tank Process}
As a second example we consider the quadruple tank process benchmark example from \cite{four_tank_2022}, i.e.
\begin{align}\label{eqn:4_tank_sys}
    \dot x_1  &= \frac{-a_1}{S_c}\sqrt{2 g x_1} + \frac{a_3}{S_c}\sqrt{2 g x_3} + \frac{\gamma_a}{3600 S_c} u_1, \\
    \dot x_2  &= \frac{-a_2}{S_c}\sqrt{2 g x_2} + \frac{a_4}{S_c}\sqrt{2 g x_4} + \frac{\gamma_b}{3600 S_c} u_2, \\
    \dot x_3 &= \frac{-a_3}{S_c}\sqrt{2 g x_3} + \frac{(1-\gamma_b)}{3600S_c} u_2, \\
    \dot x_4 &= \frac{-a_4}{S_c}\sqrt{2 g x_4} + \frac{(1-\gamma_a)}{3600S_c} u_1.
\end{align}
Let $\col(x_1, x_2, x_3, x_4) = \col(h_1, h_2, h_3, h_4)$ represent the water levels in the four tanks, respectively. The tank cross-sectional area is given by $S_c = 0.06\,\text{m}^2$. The leakage orifice areas are $a_1 = 1.31 \times 10^{-4}\,\text{m}^2$, $a_2 = 1.51 \times 10^{-4}\,\text{m}^2$, $a_3 = 9.27 \times 10^{-5}\,\text{m}^2$, and $a_4 = 8.82 \times 10^{-5}\,\text{m}^2$. The three-way valve opening ratios are given by $\gamma_a = 0.3$ and $\gamma_b = 0.4$ and $g=9.81$m/s$^2$ is the gravitational acceleration. We assume that all water tank levels are measured, i.e. $y = \col(x_1, x_2, x_3, x_4)$. Note that the system has $n_y=4$ outputs and $n_u=2$. Hence, the quadruple tank process is a MIMO system and for this reason we will compare the performance of the MS-DeepONet with the MimoONet. An open-loop identification experiment is performed using a piecewise constant input signal with a sampling time of $T_s=5$s and $8000$ samples. The datasets are generated as explained in Algorithm~\ref{alg:1}. A total of $40000$ epochs is used for training the models. 
 
 Both MPC schemes (Problems~\ref{prob:MS_DeepONet} and~\ref{prob:DMPC_DeePONet_standard}) are implemented with a prediction horizon $N=20$, loss weights $Q = 100I_{n_x}$, $R=I_{n_u}$ with input constraints $\mathbb{U}=\{u\in\mathbb{R}^2 | \col(0,0) \leq u \leq \col(4,4)\}$. Consider the closed--loop response for tracking a piecewise constant reference signal in \figurename~\ref{fig:control_4tank}. In general, both MPC controllers are able to track the reference signal well. However, it seems that the MS-DeepONet based controller has a overall better tracking performance than the standard DeepONet based controller which is confirmed by the absolute mean tracking error shown in Table~\ref{tab:performance_all}. This is to be expected since both the training and validation loss of the MS-DeepONet architecture are lower. Notably, the average computation time per control action for both controllers ($0.1678$s vs. $0.1631$s) is much shorter than the sampling period $T_s = 5$s, highlighting the architecture's potential for real-time MPC implementation. It is worth noting that the MS-DeepONet-based controller achieved performance comparable to that of the model-based MPC controller. These results are omitted from the paper due to space limitations. Lastly, note that it takes significantly longer to train the standard DeepONet model than the MS-DeepONet model ($164$min vs. $40$min) for the exact same data size and number of epochs. This simulation confirms the findings from the previous example: MS-DeepONet consistently outperforms the standard DeepONet. 
\begin{figure}[!t]
    \centering
    \includegraphics[width=0.95\columnwidth]{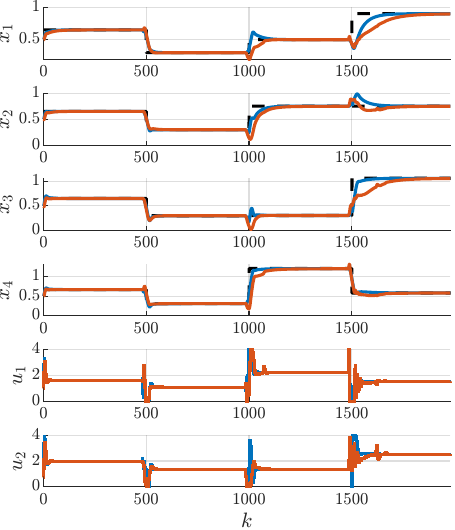}
    \caption{Reference tracking for MS-DeepONet MPC (\textcolor{Matlab_blue}{--}) vs standard DeepONet MPC (\textcolor{Matlab_orange}{--}) for Quadruple tank process.}
    \label{fig:control_4tank}
\end{figure}

\paragraph{Pendulum on a cart}
As a third example we consider the pendulum on a cart system from \cite{chatterjee2002swing}, i.e.
\begin{align}\label{eqn:pendulum_on_cart}
    \dot x_1  &= x_3, \\
    \dot x_2  &= x_4, \\
    \dot x_3 &= \frac{\frac{u}{M+m} - \frac{3 m g \sin(x_2)\cos(x_2)}{4(M+m)} + \frac{mlx_4^2\sin(x_2)}{M+m}}{1-\frac{3m\cos^2(x_2)}{4(M+m)}}, \\
    \dot x_4 &= \frac{\frac{3g\sin(x_2)}{4l} - \frac{3u\cos(x_2)}{4l(M+m)} - \frac{3mx_4^2\sin(x_2)\cos(x_2)}{4(M+m)}}{1-\frac{3m\cos^2(x_2)}{4(M+m)}},
\end{align}
where $\col(x_1,x_2,x_3,x_4) = \col(x,\theta,\dot x,\dot\theta)$ with cart position $x$ and pendulum angle $\theta$. The model parameters are $M=2.4$kg, $m=0.23$kg, $l = 0.18$m, $g=9.81$ m/s$^2$. As a control objective, we stabilize the pendulum in the upward position $\theta=0$ starting from the downward position $\theta=\pi$. Due to the unstable nature of the equilibrium point $\theta=0$ we collect closed-loop data from successful swing-up trajectories. To generate the data, we use the energy-based controller from \cite{chatterjee2002swing}. In total, we use $1100$ swing-up trajectories of length $100$. Each trajectory is transformed into Hankel matrices as explained in Algorithm~\ref{alg:1} and one large data set is generated by concatenating the data matrices from all the swing up trajectories. 
\begin{figure}[!t]
    \centering
    \includegraphics[width=0.95\columnwidth]{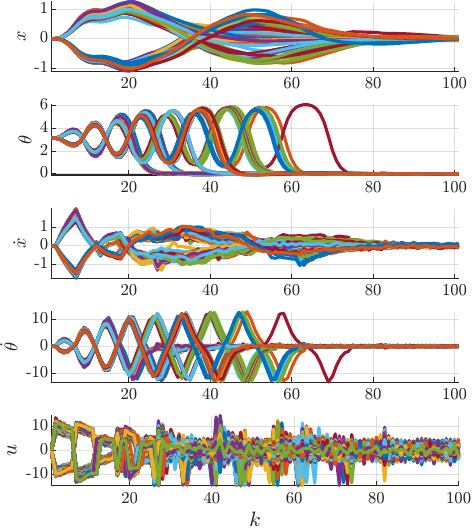}
    \caption{Randomly selected $100$ swing up trajectories from the $1000$ size dataset used for training.}
    \label{fig:swingup_data}
\end{figure}
In addition, we add a disturbance to the input such that the system is sufficiently excited, similar to the approach in \cite{näf2025choosewiselydataenabledpredictive}. Furthermore, the measurement $y \in \mathbb{R}^5$ consists of the position of the cart $x$, $\cos(\theta)$, $\sin(\theta)$, as well as the linear velocity $\dot x$ and the angular velocity $\dot \theta$ of the cart and pendulum, respectively. Some of the swing up trajectories are shown in \figurename~\ref{fig:swingup_data}. 
\begin{figure}[!t]
    \centering
    \includegraphics[width=0.95\columnwidth]{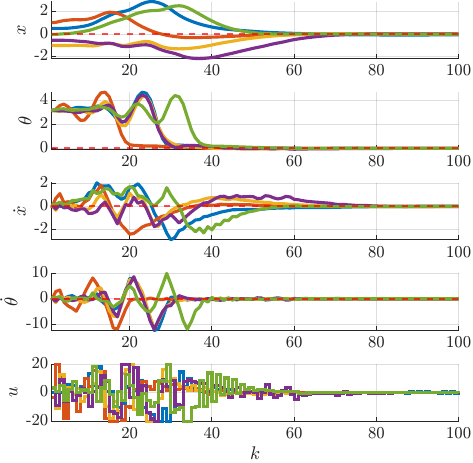}
    \caption{Swing up trajectory for pendulum on a cart for several initial conditions.}
    \label{fig:simulation}
\end{figure}
We use the MS-DeepONet to learn a prediction model. The branch and trunk networks both have $3$ hidden layers of dimensions $128$, $256$ and $128$ respectively. To control the system, we use the MPC controller defined by problem~\ref{prob:MS_DeepONet}. As a reference signal, we use $\mathbf{r} = \col(\col(0,1,0,0,0),\dots,\col(0,1,0,0,0))\in\mathbb{R}^{5N}$ and as cost weighting matrices, we use $Q = \diag(1, 1000, 1, 1, 1)$ and $R=0.01$ with a prediction horizon $N=40$, with input constraints $\mathbb{U}=\{u\in\mathbb{R} | -20 \leq u \leq 20\}$ and a sampling time of $T_s=0.1$.

From Table~\ref{tab:performance_all}, it appears that DeepONet exhibits significantly higher training and validation losses compared to the other two examples. This is expected, as the underlying dynamics in this case are considerably more nonlinear and complex. Additionally, it should be noted that due to time constraints, no extensive effort was made to optimize either the network architecture or the dataset. The dataset used in this case is also substantially larger, resulting in prolonged training times. These aspects are left for future research and optimization. Despite these limitations, \figurename~\ref{fig:simulation} demonstrates that the MS-DeepONet MPC controller is still capable of successfully swinging up and stabilizing the pendulum in the upright position from a variety of initial conditions. Importantly, these initial conditions are not present in the training set, where all swing-up trajectories originate from $x = 0$.

\section{CONCLUSIONS}\label{sec:conclusions}
In this work, we proposed the MS-DeepONet, a novel DeepONet architecture designed for approximating multi-step MIMO operators, i.e., the solution operators of continuous-time MIMO nonlinear dynamical systems. In contrast to conventional DeepONet models, the MS-DeepONet eliminates the need for separate branch networks for each input signal, while enabling the prediction of multiple future time steps across multiple outputs in a single network evaluation. We proved that the MS-DeepONet is a universal approximator for continuous-time systems with piecewise constant inputs and derived a candidate underlying basis for its architecture. This basis perspective builds on existing interpretations of neural networks, where the final hidden layer acts as an implicit adaptive basis. It highlights both structural and computational advantages of the MS-DeepONet over standard feedforward architectures, particularly in the context of predictive control. This interpretation also guides the design of data-enabled predictive controllers using MS-DeepONet models.

We evaluated the performance of the proposed MS-DeepONet-based MPC controller in three numerical benchmark examples and compared it with a standard unstacked DeepONet MPC controller. Both controllers achieved high performance with realistic computation times per control action ($0.1678$s vs. $0.1631$s for the quadruple tank process and $0.0087$s vs. $0.0078$s for the van der Pol oscillator). In terms of accuracy and efficiency, MS-DeepONet consistently outperformed the standard DeepONet, achieving lower training and validation losses, better tracking performance, and significantly shorter training times for both examples. Finally, we demonstrated the ability of the MS-DeepONet to learn and execute a swing-up trajectory for a pendulum-on-cart system, validating its potential in complex nonlinear control tasks.

Future work will explore physics-informed extensions of MS-DeepONets to enhance accuracy and generalization. We also aim to develop computationally efficient, data-enabled predictive control frameworks that take advantage of the underlying basis of the MS-DeepONet.

\bibliographystyle{IEEEtran}    
\bibliography{Neural_Operator_PC}

\begin{IEEEbiography}[{\includegraphics[width=\linewidth,height=1.25in,clip,keepaspectratio,trim=0pt 0pt 0pt 0pt]{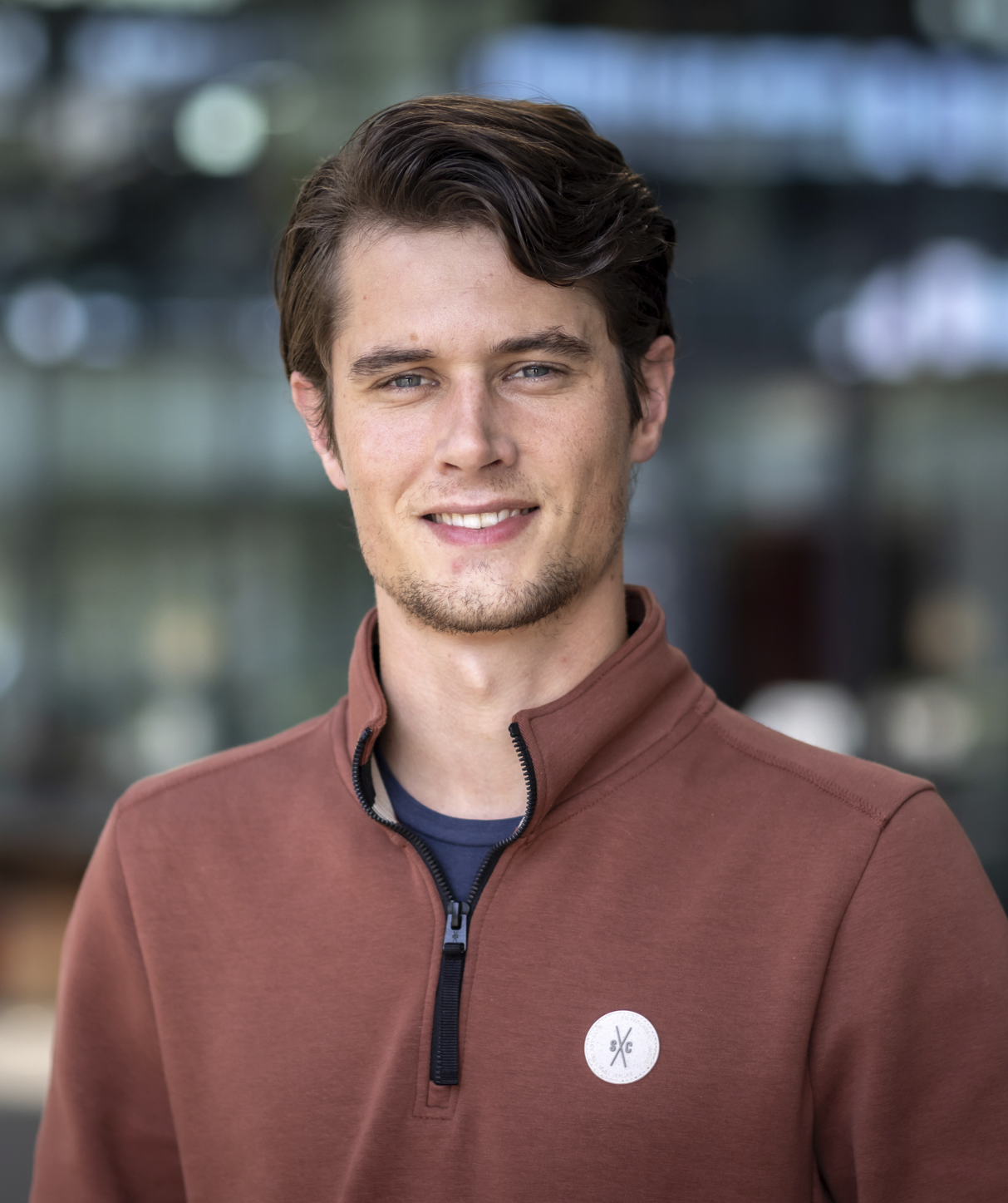}}]{Thomas O. de Jong}{\space}(Student Member, IEEE) 
received the B.S. degree in Mechanical Engineering from Eindhoven university of Technology, Eindhoven, The Netherlands, in 2021 and the M.S. degree in Systems and Control from Eindhoven university of Technology, Eindhoven, The Netherlands, in 2023. He is currently pursuing the Ph.D. degree in Systems and Control in Eindhoven, The Netherlands, in 2021 and the M.S. degree in Systems and Control at Eindhoven University of Technology, Eindhoven, The Netherlands.
\end{IEEEbiography}

\begin{IEEEbiography}
[{\includegraphics[width=\linewidth,height=1.25in,clip,keepaspectratio,trim=500pt 900pt 300pt 200pt]{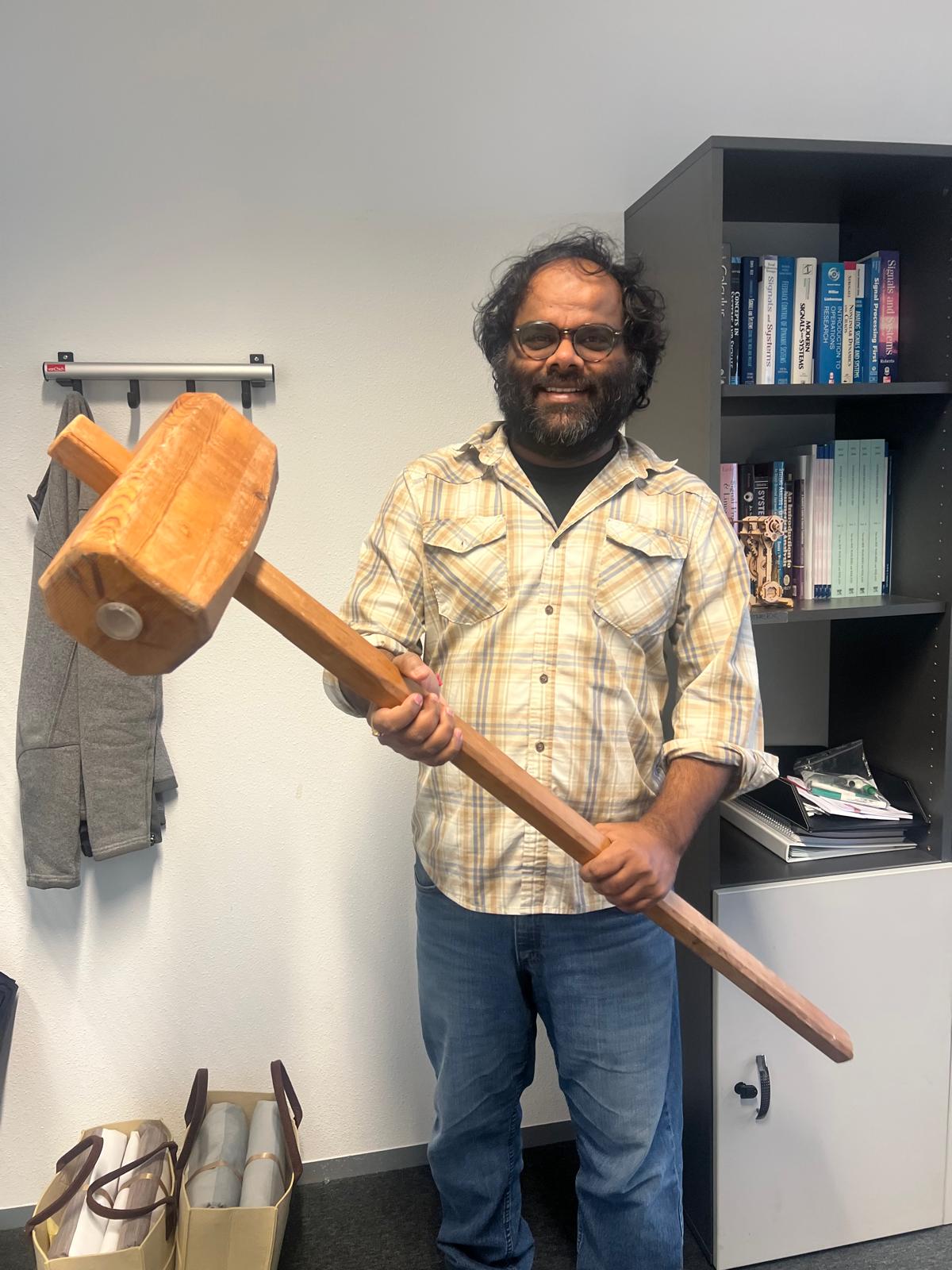}}]{Dr. Khemraj Shukla}{\space} received his Ph.D. degree in computational geophysics. In his Ph.D., he studied high-order numerical methods for hyperbolic systems and finished his research works with GMIG Group of Rice University. He is an associate professor (research) in the Division of Applied Mathematics at Brown University, Providence, Rhode Island, 02906, USA. His research focuses on the development of scalable codes on heterogeneous computing architectures. 
\end{IEEEbiography}

\begin{IEEEbiography}[{\includegraphics[width=\linewidth,height=1.25in,clip,keepaspectratio,trim=0pt 0pt 0pt 0pt]{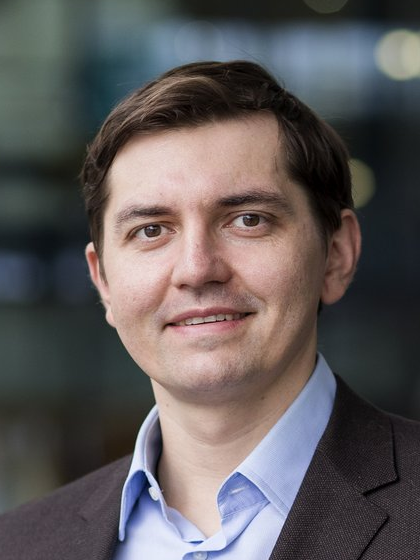}}]{Dr. Mircea Lazar}{\space}(Senior Member, IEEE)  is an Associate Professor in Constrained Control of Complex Systems at the Electrical Engineering Department, Eindhoven University of Technology, The Netherlands. His research interests cover Lyapunov functions, distributed control, neural networks and constrained control of nonlinear and hybrid systems, including model predictive control. His research is driven by control problems in power systems, power electronics, high-precision mechatronics, automotive and biological systems. Dr. Lazar received the European Embedded Control Institute Ph.D. Award in 2007 for his Ph.D. dissertation. He received a VENI personal grant from the Dutch Research Council (NWO) in 2008 and he supervised 13 Ph.D. researchers (2 received the Cum Laude distinction). Dr. Lazar chaired the 4th IFAC Conference on Nonlinear Model Predictive Control in 2012. He is an Active Member of the IFAC Technical Committees 1.3 Discrete Event and Hybrid Systems, 2.3 Nonlinear Control Systems and an Associate Editor of IEEE Transactions on Automatic Control.
\end{IEEEbiography}

\section*{APPENDIX}
In this appendix, we present additional results that are used throughout the paper for completeness. We begin by reviewing Theorem 3 from \cite{Chen_ONN}, which serves as a foundational result for establishing the adapted universal approximation Theorem~\ref{theorem:universal_approx_2}.

\begin{theorem}[Theorem 3, \cite{Chen_ONN}]\label{theorem_2}
    Suppose $K$ is a compact set in $\mathbb{R}^n$, $U$ is a compact set in $C(k)$, $g\in(TW)$, then for any $\epsilon>0$, there exist a positive integer $N$, real numbers $\theta_i$, vectors $\omega_i\in\mathbb{R}^n$, $i=1,\dots,N$, which are independent of $f\in C(K)$ and constants $c_i(f)$, $i=1,\dots,N$ depending on $f$ such that
    \begin{equation}
        |f(x) - \sum_{i=1}^Nc_i(f) g(\omega_i\cdot x + \theta_i)|<\epsilon
    \end{equation}
    holds for all $x\in K$ and $f\in U$. Moreover, each $ c_i(f)$ is a linear continuous functional defined on $U$.
\end{theorem}

Next, we present a proposition demonstrating that a stacked DeepONet architecture can be recovered from its unstacked counterpart. While the result is straightforward, it is not explicitly stated in \cite{DeepONet}, and we include it here for completeness and clarity. Given a fixed input $u = \col(u(s_1),\dots,u(s_m))\in\mathbb{R}^m$, the output equation $y(z)$ of the standard unstacked DeepONet from \figurename~\ref{fig:DeePONet_unstacked} is given by:
\begin{align}\label{eqn:branch_trunk_unstacked}
    y(z) = \sum_{k=1}^p \left( \sum_{j=1}^{n_b} w_{j}^{k} \phi_b^j(u) + \xi^k\right)   \left( \sum_{r=1}^{n_t} \alpha_{r}^k \phi^r_t(z) + b^k \right), 
\end{align}
where $\phi^j_b(u)$ and $\phi^r_t(z)$ denote the outputs of the last hidden layer of the branch and trunk networks respectively, where $j = 1,\dots, n_b$, $r = 1,\dots, n_t$ and $n_b,n_t\in\mathbb{N}$ are the number of neurons in the last layer of the hidden layers of the branch and trunk networks respectively. The output layers, denoted by the gray circles in \figurename~\ref{fig:DeePONet} and \figurename~\ref{fig:DeePONet_unstacked}, always have the same number of neurons for the branch and trunk nets, denoted by $p$. In the following proposition, we show that the one-layered neural network \eqref{eqn:universal_approx_trm} from the universal approximation Theorem, is a specific instance of the unstacked DeepONet. 
\begin{proposition}\label{proposition:stacked_unstacked}
    The universal approximation theorem for operators, as stated in Theorem \ref{theorem:universal_approx}, holds for the unstacked DeepONet \eqref{eqn:branch_trunk_unstacked}.
\end{proposition}
\begin{proof}
    For the output equation of the unstacked DeepONet \eqref{eqn:branch_trunk_unstacked}, set $\alpha_r^k = 1$ if $k=r$ and $\alpha_r^k = 0$ if $k\neq r$ and set $\xi_k = b^k = 0$ for all $k = 1,\dots,p$. This can be interpreted as removing the output layer of the trunk network and removing the bias terms in the branch network. Then \eqref{eqn:branch_trunk_unstacked} becomes:
    $$
    y(z) = \sum_{k=1}^p  \sum_{j=1}^{n_b} w_{j}^{k} \phi_b^j(u)  \phi^k_t(z),
    $$
    Next, choose $\phi_t^k = \sigma(w_k \cdot u + \zeta_k)$ and $\phi_b^j = \sigma(\sum_{i=1}^m\xi_{ij} u(s_i)+ \theta_j)$, such that 
    $$
    y(z) = \sum_{k=1}^p  \sum_{j=1}^{n_b} w_{j}^{k} \sigma\left(\sum_{i=1}^m\xi_{i,j} u(s_i) + \theta_j\right)  \sigma(w_k \cdot z + \zeta_k),
    $$
    to convert the branch and trunk networks into single layer networks. Choose the number of neurons in the last layer of the branch network $n_b$ is at least larger than the number of neurons in the output layer $p$ such that there exists an $n\in\mathbb{N}$ with $n_b \geq np$. Next set $w_{j}^{k}=0$, $\xi_{i,j}=0$ and $\theta_j=0$ for all $j>np$ where $n\in\mathbb{N}$ is the largest natural number such that $np\leq n_b$ and define $w_{r,q}^k = w^k_{(q-1)n+r}$, $\xi_{i,r}^q = \xi_{i,(q-1)n+r}$ and $\theta_{r}^q = \theta_{(q-1)n+r}$, where $q=1,\dots,p$ and $r=1,\dots,n$ such that:
    \begin{align*}
        &y(z) =\\& \sum_{k=1}^p \sum_{r=1}^{n} \sum_{q=1}^{p} w_{r,q}^{k} \sigma\left(\sum_{i=1}^m \xi_{i,r}^q \cdot u(s_i) + \theta_r^q\right)  \sigma(w_k \cdot z + \zeta_k).
    \end{align*}
    This step is simply a change of index names which makes the next step easier. Finally, set $w^k_{r,q} = 0$ if $q \neq k$ and define $w_r^k = w_{r,k}^k$, such that we obtain:
    $$
        y(z) = \sum_{k=1}^p \sum_{r=1}^{n} w_{r}^{k} \sigma\left(\sum_{i=1}^m \xi_{i,r}^k \cdot u + \theta_r^k\right)  \sigma(w_k \cdot z + \zeta_k),
    $$
    which again amounts to removing certain connections between neurons. This implies that the output of the unstacked DeepONet recovers the outputs of the stacked DeepONet for a specific selection of parameters, which yields the claim via Theorem~\ref{theorem:universal_approx}.
\end{proof}
\end{document}